\documentclass[11pt,twoside]{article}
\usepackage[utf8]{inputenc}

\usepackage{colortbl}

\usepackage{graphicx}
\usepackage[active]{srcltx}

\usepackage{amsmath,amssymb}
\usepackage{bm}
\usepackage{amsthm}

\usepackage[all,2cell]{xy}
\usepackage{mathrsfs}
\usepackage{mathtools}
\mathtoolsset{showonlyrefs}

\usepackage{tikz-cd}
\usetikzlibrary{cd}

\usepackage{xcolor}

\usepackage[shortlabels]{enumitem}
\setlist[itemize,enumerate]{
    itemsep=0.5ex,     
    topsep=3pt,    
    partopsep=3pt,  
    parsep=0pt
}

\usepackage[a4paper, left=25mm, right=25mm, top=25mm, bottom=25mm]{geometry}

\usepackage{marginnote}

\usepackage{imakeidx}
\makeindex[intoc]

\usepackage[nottoc]{tocbibind}
\usepackage[colorinlistoftodos]{todonotes} 

\usepackage[colorlinks=true,pagebackref=false]{hyperref}
\hypersetup{urlcolor=blue, citecolor=red, linkcolor=blue}

\theoremstyle{plain}
\newtheorem{theorem}{Theorem}[section]
\newtheorem{proposition}[theorem]{Proposition}
\newtheorem{corollary}[theorem]{Corollary}
\newtheorem{lemma}[theorem]{Lemma}

\theoremstyle{definition}
\newtheorem{definition}[theorem]{Definition}
\newtheorem{remark}[theorem]{Remark}
\newtheorem{example}[theorem]{Example}

\AtBeginEnvironment{remark}{%
  \pushQED{\qed}%
}
\AtEndEnvironment{remark}{\popQED\endexample}

\AtBeginEnvironment{example}{%
  \pushQED{\qed}%
}
\AtEndEnvironment{example}{\popQED\endexample}

\newcommand\restr[2]{{
  \left.\kern-\nulldelimiterspace 
  #1 
  \right|_{#2} 
}}

\newcommand{\R}{\mathbb{R}}

\renewcommand{\d}{\mathrm{d}}

\newcommand{\Cinfty}{\mathscr{C}^\infty}
\newcommand{\T}{{\mathrm T}}
\let\Tan\T
\newcommand{\cT}{\T^{\ast}}
\newcommand{\Id}{\mathrm{Id}}

\newcommand*{\inn}[1]{\iota_{#1}}

\newcommand{\Lie}{\mathscr{L}}

\newcommand{\X}{\mathfrak{X}}

\newcommand{\parder}[2]{\frac{\partial #1}{\partial #2}}
\newcommand{\dparder}[2]{\dfrac{\partial #1}{\partial #2}}

\DeclareMathOperator{\Ima}{Im}
\DeclareMathOperator{\rk}{rank}
\DeclareMathOperator{\cl}{cl}

\DeclareMathOperator{\codim}{codim}

\DeclareMathAlphabet{\mathpzc}{OT1}{pzc}{m}{it}
\def\d{\mathrm{d}}

\newcommand\dfn[1]{{\textbf{\emph{#1}}}}

\newcommand{\fracp}[2]{\frac{\partial #1}{\partial #2}}

\def\X{\mathfrak{X}}
\def\d{\mathrm{d}}

\def\Im{\mathrm{Im \,}}
\def\Ker{\mathrm{Ker \,}}
\def\Tan{\mathrm{T}}
\def\Lie{\mathscr{L}}

\def\K{\mathcal{K}}
\def\longto{\longrightarrow}

\usepackage{fancyhdr}

\begin{document}
\parskip=3pt

{\huge\sffamily\raggedright
Pairs of differential forms:\\[10pt]
a framework for precontact geometry
}

\vspace{1em}

{\large\raggedright\sffamily
4 February 2026
}

\vspace{1em}

{\Large\raggedright\sffamily
    Xavier Gràcia
}\vspace{1mm}\newline
{\raggedright
    Dept.\ of Mathematics, Universitat Politècnica de Catalunya, Barcelona\\
    e-mail: 
    \href{mailto:xavier.gracia@upc.edu}{xavier.gracia@upc.edu} --- {\sc orcid}: 
    \href{https://orcid.org/0000-0003-1006-4086}{0000-0003-1006-4086}
}

\medskip

{\Large\raggedright\sffamily
    Ángel Martínez-Muñoz
}\vspace{1mm}\newline
{\raggedright
    Dept.\ of Computer Engineering and Mathematics, Universitat Rovira i Virgili, Tarragona\\
    e-mail: \href{mailto:angel.martinezm@urv.cat}{angel.martinezm@urv.cat} --- {\sc orcid}: \href{https://orcid.org/0009-0002-8944-5403}{0009-0002-8944-5403}
}

\medskip

{\Large\raggedright\sffamily
    Xavier Rivas
}\vspace{1mm}\newline
{\raggedright
    Dept.\ of Computer Engineering and Mathematics, Universitat Rovira i Virgili, Tarragona\\
    e-mail: \href{mailto:xavier.rivas@urv.cat}{xavier.rivas@urv.cat} --- {\sc orcid}: \href{https://orcid.org/0000-0002-4175-5157}{0000-0002-4175-5157}
}

\vspace{.7em}

{\large\bf\raggedright
    Abstract
}\vspace{1mm}\newline
{\parindent 0pt
Precontact manifolds extend contact geometry by weakening the maximal non-integrability condition of the defining $1$-form. 
We clarify the geometric foundations of this structure by studying general pairs of a $1$-form and a $2$-form under mild regularity conditions. 
We characterize them through their class, analyse the role of distinguished vector fields, such as Reeb or Liouville fields, and study other associated geometrical objects. 
Precontact structures are then treated as the special case of pairs formed by a nowhere-vanishing $1$-form and its exterior derivative. 
We also define Hamiltonian dynamics on precontact manifolds.
Several examples are presented to illustrate the theory.
}
\smallskip

\medskip


\medskip


{\setcounter{tocdepth}{2}
\def\baselinestretch{0.79}
\small
\def\addvspace#1{\vskip 1pt}
\parskip 0pt plus 0.1mm
\tableofcontents
}

\newpage

\pagestyle{fancy}

\fancyhead[LE]{Pairs of differential forms: a framework for precontact geometry} 
\fancyhead[RE]{}
\fancyhead[RO]{X. Gràcia, À. Martínez-Muñoz and X. Rivas} 
\fancyhead[LO]{}    

\fancyfoot[L]{}     
\fancyfoot[C]{\thepage}                  
\fancyfoot[R]{}            

\setlength{\headheight}{17pt}

\renewcommand{\headrulewidth}{0.1pt}  
\renewcommand{\footrulewidth}{0pt}    

\renewcommand{\headrule}{%
    \vspace{3pt}                
    \hrule width\headwidth height 0.4pt 
    \vspace{0pt}                
}

\setlength{\headsep}{30pt}  

\section{Introduction}

In the late 19th century, Sophus Lie introduced the concept of contact transformations, a class of transformations that preserve the solution sets of certain differential equations. 
His work laid the groundwork for the modern theory of contact geometry, which has since found applications in a wide range of fields, including geometric optics, thermodynamics, Hamiltonian dynamics, and fluid mechanics. 
For a detailed historical account of the development of contact geometry and topology, as well as a thorough bibliography, see~\cite{Gei_01}.

A \emph{contact distribution} on a $(2n{+}1)$-dimensional manifold is a maximally non-integrable hyperplane field. 
Locally, this hyperplane field can be expressed as the kernel of a so-called \emph{contact form}~$\eta\in \Omega^1(M)$ satisfying the condition
\[
\eta \wedge (\d\eta)^n \neq 0
\]
at every point. 
This condition ensures that the distribution defined by $\Ker\eta$ is as far from being integrable as possible, in the sense of Frobenius' theorem. 
Note that, if $f \in \Cinfty(M)$ is a nowhere-vanishing function, 
then $f\eta \in \Omega^1(M)$ generates the same contact distribution as $\eta$.
Thus, a contact distribution is locally equivalent to a \emph{conformal class} of $1$-forms satisfying the preceding non-integrability condition. 

The tools and structures of contact geometry have proven to be very effective to model dissipative systems~(see for instance~\cite{Bra_18,Got_16,RMS_17} for some applications). 
In particular, for mechanical systems, contact geometry has been successfully applied to describe non-conservative Lagrangian and Hamiltonian dynamics~\cite{Lainz_19,GGMRR_21}. 
The variational formulation of such systems is due to G.~Herglotz (see~\cite{Herglotz1930,Herglotz1985}), around 1930. 
The approach is analogous to the formulation of autonomous and non-autonomous Hamiltonian systems via, respectively, symplectic and cosymplectic geometry~\cite{MR_2025}.

Within the framework of contact Lagrangian mechanics, there is a topic that has received comparatively little attention: 
the analysis of \emph{singular} Lagrangian systems. 
For such systems, the underlying geometrical framework is not that of contact geometry, as the structures involved exhibit some degeneracy. 
This situation parallels that of ordinary (symplectic) mechanics, where the geometric study of singular Lagrangian functions naturally involves that of presymplectic geometry (see~\cite{Car_1990, GNH_1978}).

A first attempt to study singular Lagrangian systems in the context of contact mechanics can be found in~\cite{LL_19a}, where the authors give a definition of \emph{precontact form}. 
As later pointed out in~\cite{LGGMR_23}, this analysis 
is not general enough
since it only considers the case in which the geometric structure admits a Reeb vector field, 
an assumption that is not always satisfied.
The main reason for this limitation is that \cite{LL_19a} was aimed at the study of Hamiltonian systems on precontact manifolds, and the equations of motion explicitly relied on the existence of a Reeb vector field. 
However, in a recent paper~\cite{GMRR_2025}, we derived two new equivalent formulations of the contact Hamiltonian equations. 
Notably, these do not depend on the Reeb vector field, whose existence and uniqueness are not guaranteed.
This development naturally motivates the study of precontact forms in full generality, which is precisely the goal of the present work.

Precontact structures defined as hyperplane fields have also been studied, see for instance~\cite{GG_23,Tor_2018,Vit_2018}.  
In contrast, our approach focuses on the case where the precontact structure is given by a global $1$-form%
\footnote{In the literature this case is sometimes referred to as \emph{co-oriented} (pre)contact structures.}. 
This more specific viewpoint is the usual one in some applications,
as for instance geometric mechanics,
but will also be key for our study,
as we are going to explain.

Similarly to the definition of contact form,
we will define a \emph{precontact form} to be a nowhere-vanishing 1-form 
$\eta\in \Omega^1(M)$ such that the rank of $\d\eta$ is constant and, for a certain~$r$,
$$    
\eta\wedge(\d\eta)^r \neq 0\,,
\qquad
\eta\wedge(\d\eta)^{r+1} = 0
$$
at every point of~$M$. 
This condition is related to the notion of \emph{class of a differential form} (see~\cite[Ch.\,VI]{god_1969}). 
The class of a differential form, when constant, intuitively represents the minimum number of coordinates required to express the form in a local.
It can be computed as an algebraic condition involving the form and its differential.
Then, for a 1-form $\eta$, it is a precontact form if and only if it has class $2r{+}1$ or $2r{+}2$, provided $\eta$ is nowhere-vanishing.

To establish these and other results, we have found it useful to take a slightly more general approach, 
by considering a pair $(\tau,\omega)$ of a differential $1$-form and a $2$-form.
By generalizing the definition of class to such pairs, we can study the properties of distinguished vector fields, such as Reeb and Liouville vector fields, whose existence is equivalent to the class of the pair $(\tau,\omega)$ being odd or even, respectively. 
This characterizations provide the necessary tools to describe the properties of some objects naturally induced by the geometric structure, 
like the tensor field $\omega + \tau \otimes \tau$ or the $3$-form $\tau \wedge \omega$.
We also establish the conditions under which certain functions can be used to change the parity of the class.
All of these results are applied to the particular pair $(\eta,\d\eta)$, formed by a $1$-form $\eta$ and its exterior derivative, which is precisely the case of precontact manifolds. 
As an application, Hamiltonian dynamics on precontact manifolds is introduced.


The paper is organized as follows.
In Section~\ref{ch_2}, we begin by introducing the notations and definitions of the main objects used throughout the paper and state some basic results. 
Then, we study pairs of a differential $1$-form and a $2$-form.
We characterize such structures in terms of their class, and relate this notion to the existence of Reeb vector fields and Liouville vector fields.
In Section~\ref{ch_3} we introduce and study the properties of some geometrical objects induced by pairs of differentials forms.
Some relevant examples of the framework are presented in Section~\ref{sec_examples}. 
In Section~\ref{sec_precontact}, we apply the previous results to characterize and study precontact manifolds
We also define Hamiltonian mechanics on precontact manifolds, and provide some necessary conditions that have to be satisfied in order for solutions to exist.
Finally, in Section~\ref{sec_changes}, we present some novel results which characterize the functions that can alter the parity of the class of a pair of a $1$-form and a $2$-form. 
We also provide sufficient conditions for the existence of such functions.

Throughout the paper, all the manifolds and maps are smooth.
Sum over crossed repeated indices is understood.

\section{Doublets of a one-form and a two-form} 
\label{ch_2}

This section is devoted to the study of several properties 
of a pair formed by a $1$-form and a $2$-form,
subject to some regularity conditions.
We start with the preliminary definitions and then we state and prove some essential results.

\subsection{Rank of a differential form}
\label{sec_2.1}

\begin{definition}
Let $\alpha \in \Omega^p (M)$ be a $p$-form. 
It defines a vector bundle morphism
$$
\begin{array}{rccc}
     \widehat\alpha\colon & \Tan M & \longto & \bigwedge\nolimits^{p-1}\cT M \\
     & v & \longmapsto & \inn{v}\alpha
\end{array}
$$
whose kernel at a point $x$ is 
$$
\Ker \alpha_x = \{ v\in \Tan_x M \mid \inn{v} \alpha_x =0  \}\,.
$$
Its annihilator is 
$$
(\Ker \alpha_x) ^{\circ} = 
\{  u\in \Tan^*_x M\mid \inn{v}u=0\ \ \text{ for every }v\in \Ker \alpha_x\} \,.
$$
The rank of~$\alpha$ is the rank of this linear map;
it coincides with the codimension of the kernel.
\end{definition}

\begin{theorem}[\cite{god_1969}]
\label{theorem_annihilator}
Given a $p$-form $\alpha \in \Omega^{p} (M)$ on a manifold $M$, the image of the multilinear map
\vadjust{\kern -1em}
\begin{align}
\Tan_xM \times \overset{(p-1)}{\dotsb} \times \Tan_xM &\longrightarrow 
\Tan_x^*M 
\\
(v_1,\dotsc,v_{p-1}) &\longmapsto 
\inn{v_1} \dotsb \inn{v_{p-1}} \alpha
\end{align}
is~$(\Ker \alpha_x)^{\circ}$.
\end{theorem}

In particular, this theorem implies that,
for a $1$-form $\alpha\in \Omega^{1} (M)$, we have
$$
(\Ker \alpha_x)^{\circ} = \langle \alpha_x\rangle\,.
$$ 
And if $\omega \in\Omega^2 (M)$ is a $2$-form, then 
$$
(\Ker\omega_x)^{\circ} =  \langle \inn{e_1}\omega_x,\dotsc,\inn{e_m}\omega_x \rangle\,,
$$
where $e_1,\dotsc, e_m$ constitute a basis of $\Tan_xM$.

\begin{theorem}[\cite{god_1969,LM_1987}]
\label{thm:even_rank}
Let $\omega \in \Omega ^2(M)$ be a $2$-form on $M$.
Then the rank of $\omega$ is even at every point of $M$. 
Also, this rank is $2r$ at a point~$x$ if and only if
$$
(\omega_x)^r\neq0\,, \qquad 
(\omega_x)^{r+1}=0\,.
$$
\end{theorem}
(We denote by $\omega^r$ the exterior product of $r$ copies of $\omega$.)

\begin{proposition}[\cite{god_1969}]\label{rank_powers}
    Let $\omega \in \Omega^2(M)$ be a $2$-form of constant rank. Then, if $\omega$ has rank~$2r$, we have
    $$(\Ker \omega)^{\circ}=(\Ker \omega^s)^{\circ}\,,$$
    for all $s=1,\dotsc,r$.
\end{proposition}

\subsection{Pairs of differential forms}

\begin{definition}
Let $\alpha\in \Omega^p(M)$ be a $p$-form on $M$ and $\beta \in \Omega^q(M)$. 
The \dfn{characteristic distribution} of the pair $(\alpha, \beta)$ at a point $x\in M$ is 
$$
\K_{(\alpha,\beta)_x} = \Ker \alpha_x \cap \Ker \beta_x \,.
$$
The \dfn{class} of $(\alpha, \beta)$ at~$x$ is the codimension of its characteristic distribution at~$x$.
$$
\cl(\alpha,\beta)_x = 
\codim \K_{(\alpha,\beta)_x} 
\,.
$$
\end{definition}

The annihilator of the characteristic distribution is
$$
    \K_{(\alpha,\beta)_x}^{\circ} =
    (\Ker \alpha_x \cap \Ker\beta_x)^{\circ} =
    (\Ker \alpha_x)^\circ + (\Ker\beta_x)^{\circ}
\,,
$$
therefore the class is equal to the dimension of this annihilator,
$\cl(\alpha,\beta)_x = 
\dim \left( (\Ker \alpha_x)^\circ + (\Ker\beta_x)^{\circ} \right)
$.

By the Grassmann identity we have
$$
\dim ((\Ker \alpha )^{\circ}+(\Ker \beta)^{\circ}) 
=
\dim (\Ker \alpha)^{\circ} +
\dim (\Ker \beta)^{\circ} -
\dim ((\Ker \alpha)^{\circ}\cap(\Ker \beta)^{\circ}) 
\,,
$$
or, in other words,
\begin{equation}
\label{eq:grassmann_rank}
\cl(\alpha,\beta)
=
\rk \alpha +
\rk \beta -
\dim ((\Ker \alpha)^{\circ}\cap(\Ker \beta)^{\circ}) 
\,.
\end{equation}

\begin{remark}
\label{rmrk:class}
Given a $p$-form $\alpha \in \Omega^p(M)$, the usual definition of characteristic distribution is
$$
\K \coloneqq \K_{(\alpha,\d\alpha)} = 
\Ker \alpha \cap \Ker \d \alpha\,,
$$
and the class of $\alpha$ is the codimension of the characteristic distribution~\cite{god_1969,LM_1987}.

If the form $\alpha\in\Omega^p(M)$ is closed, then 
$$
\K= \Ker \alpha\,,
$$
because $\d\alpha = 0$. 
This, in particular, implies that if $\alpha$ is a closed $p$-form, then the rank of $\alpha$ is equal to its class.
Thus, by Theorem \ref{thm:even_rank}, every closed $2$-form has even class.
\end{remark}

\subsection{Doublets of a 1-form and a 2-form}
\label{sec_pairs}

From now on we will consider a pair $(\tau,\omega)$, 
with $\tau \in \Omega^1(M)$ and $\omega \in \Omega^2(M)$. 
For most of our results we will need these forms to satisfy some regularity conditions leading to constant rank distributions. 
Prior to this, let us explore some situations where nonconstant rank might arise.

First, if $\tau$ is a 1-form then 
its rank at a point $x$ is either $0$ if $\tau_x=0$, 
or $1$ if $\tau_x \neq 0$.
\begin{itemize}
\item 
If $\tau_x = 0$, then 
$\codim (\Ker \tau_x) = \dim (\Ker \tau_x)^\circ = 0$, and 
$\dim \left( \strut
(\Ker \tau_x)^{\circ}\cap(\Ker \omega_x)^{\circ}
\right) =0$. 
So equation~\eqref{eq:grassmann_rank} reduces to
\begin{equation}
\label{eq:grassman_even}
\cl(\tau_x,\omega_x) = \rk \omega_x\,,   
\end{equation}
and so the class is even.
\item  
If $\tau_x \neq 0$, then 
$(\Ker \tau_x)^\circ = \langle \tau_x \rangle$, therefore  
$\codim (\Ker \tau_x) = \dim (\Ker \tau_x)^\circ = 1$.
From this,
$$
\dim \left( \strut
(\Ker \tau_x)^{\circ}\cap(\Ker \omega_x)^{\circ}
\right) 
\in \{0,1\}
.
$$
Moreover, equation~\eqref{eq:grassmann_rank} reduces to
\begin{equation}
\label{eq:grassmann_doublets}
\cl(\tau_x,\omega_x)
=
1 +
\rk \omega_x -
\dim ((\Ker \tau_x)^{\circ}\cap(\Ker \omega_x)^{\circ}) 
\,.
\end{equation}
\end{itemize}

\begin{proposition}
\label{prop:doublet_conditions}
Consider a pair $(\tau,\omega)$.
The class at a point~$x$ is $2r$ or $2r+1$ if, and only if, the rank of $\omega_x$ is $2r$.
Moreover:
\begin{enumerate}[$(a)$]
\item 
If the class of a pair is constant, 
then the rank of $\omega$ is constant. 
\item 
If the class is constant and odd,
then $\tau$ is nowhere-vanishing.
\end{enumerate} 
    
\end{proposition}
\begin{proof}
If the rank of $ \omega_x$ is $2r$, then it is clear from the previous comments that the class must be $2r$ or $2r+1$. 

Assume that $\cl(\tau_x,\omega_x)=2r$.
Equation~\eqref{eq:grassmann_rank}
implies that $\rk(\omega_x)=2r - \rk{\tau_x} +
\dim ((\Ker \tau_x)^{\circ}\cap(\Ker \omega_x)^{\circ})$,
and the last two terms are either 0 or~1, 
so they cannot sum up to be bigger than $1$ or smaller than $-1$.
Hence, since the rank of $\omega$ at $x$ must be even, $\rk (\omega_x)= 2r$.

If $\cl(\tau_x,\omega_x)=2r+1$, then it follows from equation~\eqref{eq:grassman_even} that  $\tau_x\neq 0$. Now, using equation~\eqref{eq:grassmann_doublets}, and the fact that 
$
\dim \left( \strut
(\Ker \tau_x)^{\circ}\cap(\Ker \omega_x)^{\circ}
\right) 
\in \{0,1\}
$, we have $\rk (\omega_x)= 2r$.

\end{proof}

The converse of the previous proposition is not true, as shown in the following counterexamples.
    
\begin{example}
Constant rank of the forms does not guarantee that the pair $(\tau,\omega)$ has constant class. 
Consider $M=\mathbb{R}^3$ with coordinates $(x,y,z)$, and let
$$
    \tau = \d z\,, \qquad 
    \omega = \d y \wedge \d z + z\, \d x \wedge \d y\,.
$$
Both forms have constant rank ($\rk\tau = 1$, $\rk\omega = 2$), but their class is not constant: 
it is $3$ for $z\neq 0$ and drops to $2$ at $z=0$. 
\end{example}

\begin{example} 
If the rank of $\omega$ is constant, and the class of the pair is constant and even, then $\tau$ is not necessarily nowhere-vanishing. As an example, we take $M= \R^2$, with coordinates $(x,y)$, and
$$
    \tau = x\d x\,, \qquad 
    \omega = \d x \wedge \d y\,.
$$
Clearly $\tau$ vanishes at the points where $x=0$, $\omega$ has constant rank $2$ and the class of the pair is constant and equal to $2$.
\end{example}

From their definitions it is clear that these statements are equivalent:
(a) $\K_{(\tau,\omega)} = \Ker \tau \cap \Ker \omega = 0$;
(b) $\cl(\tau,\omega) = \dim M$;
(c) $\T^*M = (\Ker \tau)^{\circ} + (\Ker \omega)^{\circ}$.
Moreover:

\begin{proposition}
\label{prop:directsum}
Let $(\tau,\omega)$ be a pair of a 1-form and a 2-form. 
The following are equivalent:
\begin{enumerate}[$(i)$]
\item 
$\T^* M =(\Ker \tau)^{\circ} \oplus (\Ker \omega)^{\circ}$;
\item
$\T M = \Ker \tau \oplus \Ker \omega $;
\item 
$\cl (\tau,\omega) = \dim M$ and either the class is odd or $\tau=0$.
\end{enumerate}
\end{proposition}

In view of Proposition~\ref{prop:doublet_conditions} and Proposition~\ref{prop:directsum},
it is apparent that the cases where the class is constant
have nicer properties.
To avoid repetitions,
we will use the following terminology:

\begin{definition}
We call $(\tau,\omega)$ a \dfn{doublet} if the
pair has constant class.
\end{definition}

\subsection{Reeb and Liouville vector fields}
\label{sec_ReebLiouville}


In this section, we introduce two distinguished types of vector fields associated with a doublet, highlighting their existence in connection to the class. 
These vector fields are central to the subsequent sections and have significant applications in geometry and analytical mechanics.

\begin{definition}
Let $(\tau,\omega)$ be a pair of a 1-form and a 2-form.
\begin{itemize}
\item A \dfn{Reeb vector field} for the pair 
is a vector field $R$ satisfying
\begin{equation}
\label{reeb_definition}
\inn{R}\tau= 1\,, \qquad \inn{R}\omega=0\,. 
\end{equation}

\item A \dfn{Liouville vector field} for the pair 
is a vector field $\Delta$ satisfying
\begin{equation}
\label{Liouville_vf}
\inn{\Delta}\omega = \tau \,.
\end{equation}
\end{itemize}

A similar nomenclature will be applied at a specific point $x \in M$:
a vector $R_x$ is a \dfn{Reeb vector}, 
or a vector $\Delta_x$ is a \dfn{Liouville vector},
for the pair $(\tau_x,\omega_x)$,
when the same relations hold at~$x$.
\end{definition}

Notice that Reeb and Liouville vector fields cannot coexist.
Indeed, if both a
Reeb vector field~$R$ and a Liouville vector field $\Delta$ existed simultaneously, then
$$
\inn{R} \inn{\Delta} \omega
  = \inn{R} \tau
  = 1\,,
$$
but also, since $\inn{R} \omega = 0$, we have
$$
\inn{R} \inn{\Delta} \omega
  = - \inn{\Delta} \inn{R} \omega
  = 0\,,
$$
in contradiction with the preceding equation.

Then, the parity of the class gives 
necessary and sufficient conditions
for the existence of Reeb or Liouville vector fields, as stated in the next theorem.

\begin{theorem}
\phantomsection\label{thm:class-parity}
Let $(\tau,\omega)$ be a doublet.
\begin{itemize}
\item 
The following statements are equivalent:
\begin{enumerate}[$(i)$]
\item 
The class of the doublet is odd.
\item 
$\Ker \omega \not\subset \Ker\tau$.
\item 
The doublet has a Reeb vector field.
\end{enumerate}
\item 
The following statements are equivalent:
\begin{enumerate}[$(i)$]
\item 
The class of the doublet is even.
\item 
$\Ker \omega \subset \Ker \tau$.
\item 
The doublet has a Liouville vector field.
\end{enumerate}
\end{itemize}
Similar statements hold for a pair $(\tau_x,\omega_x)$
and Reeb or Liouville vectors.
\end{theorem}
\begin{proof}\,
\begin{itemize}
\item \textbf{Odd case}.
Assume that the class of the doublet is odd. 
It follows from equation~\eqref{eq:grassmann_doublets} that this is equivalent to the intersection of the annihilators of the kernels being zero and also to
$$
(\Ker \tau)^{\circ} \not\subset (\Ker \omega)^{\circ}\,.
$$
This is the same as
\begin{equation} \label{eq:annihilator_inclusion_odd}
\Ker \omega \not\subset \Ker\tau\,.
\end{equation}
This proves the equivalence between (i) and (ii).

From the definition of Reeb vector field, 
if condition (iii) holds, 
then condition~(ii) follows directly.

Conversely, let us assume condition~(ii). 
Since $\Ker \omega\subset \T M$ has constant dimension by assumption, it defines a vector subbundle, and the restricted map
$\phi\coloneqq \tau|_{\Ker \omega} \colon \Ker \omega \longto M \times \R\,,$
is a vector bundle morphism, which is surjective thanks to condition~(ii). 
We can then apply the splitting lemma (see for instance 
\cite[p.\,118]{Die_EA3})
to the short exact sequence
$$
0 \longrightarrow \K_{(\tau,\omega)} \longrightarrow \Ker \omega \xrightarrow{\ \phi\ } M \times \R \longrightarrow 0
$$
to obtain a map $j\colon M\times \R \to \Ker \omega$ such that $\phi \circ j = \Id$. Finally, we can use $j$ to lift the constant section $s\colon (x)\in M \mapsto (x,1)\in  M\times \R$ to a section of $\Ker \omega$ that satisfies the properties of the Reeb vector field by construction.

\item \textbf{Even case}.
Assume that the class of the doublet is even. 
By the same argument as before, this is equivalent to
\begin{equation}
\label{annihilator_inclusion_even}
(\Ker \tau)^{\circ} \subset (\Ker \omega)^{\circ}
\,,
\end{equation}
which amounts to
$$
\Ker \omega \subset \Ker \tau
\,.
$$
This proves the equivalence between (i) and~(ii).

If there exists a Liouville vector field $\Delta$, then condition~(ii) follows from the fact that, if $X\in \Ker \omega$, we have
$$\inn{X} \tau = \inn{X}\inn{\Delta} \omega = -\inn{\Delta}\inn{X}\omega =0\,.$$
Conversely, assume condition (ii). 
We have 
$$
\Ima \widehat\omega =
(\Ker \widehat\omega)^\circ \supset
(\Ker \widehat\tau)^\circ = 
\langle \tau \rangle
\,,
$$
where we have used Theorem~\ref{theorem_annihilator}
and~(ii).
But since
$0\to\Ker \widehat\omega \to 
\Tan M \to
\Ima \widehat\omega\to0$
is exact it splits, and thus we have a map $j\colon \Im \widehat \omega \to \T M$, such that $\widehat \omega\circ  j =\Id$, which can then be used to construct a section $\Delta = j\circ\tau$ satisfying the conditions of a Liouville vector field.
\end{itemize}
The statements at a given point don't require any regularity condition
and their proof is indeed easier.
\end{proof}

We finish this section with a short analysis of the Lie derivative of a doublet 
with respect to the Reeb or Liouville vector fields.
Its proof is direct from the definitions:
\begin{lemma}
For a Reeb vector field $R$,
$$
\Lie_R \tau = \inn{R} \d\tau \,, \qquad \Lie_{R} \omega = \inn{R} \d\omega\,;
$$
and for a Liouville vector field $\Delta$,
$$
\Lie_{\Delta} \tau = \inn{\Delta} \d \tau\,, \qquad \Lie_{\Delta} \omega  = \inn{\Delta} \d \omega + \d \tau\,.
$$
\end{lemma}

\subsection{Other characterizations of the class}

\begin{lemma}\label{lem:odd_conditions}
    Let $(\tau,\omega)$ be a pair of a differential 1-form and a 2-form and $x\in M$.  
Then $\cl(\tau_x,\omega_x)=2r+1$ if, and only if,
    \begin{equation}
            \omega_x ^{r+1} = 0\,,
            \qquad
            \tau_x \wedge \omega_x^r \neq 0\,.
    \end{equation}
\end{lemma}
\begin{proof}
    Assume that $\cl(\tau_x ,\omega_x)=2r+1$. Then, the rank of $\omega$ must be $2r$ by Proposition~\ref{prop:doublet_conditions}, which implies that
    $$\omega_x^{r}\neq 0, \qquad
            \omega_x ^{r+1} = 0\,.$$
    Now, as the class is odd, $\tau_x\neq 0$ and there exists a Reeb vector $R_x$. If we assume that $\tau_x \wedge \omega_x ^r = 0$, we have
    $$0=\inn{R_x} (\tau_x\wedge \omega_x^{r})= \omega_x^r$$
    which is a contradiction, and thus we must have $\tau_x \wedge \omega_x^r \neq 0$.
    
    For the converse, note that the two conditions imply that $\rk(\omega_x)=2r$ and that $\tau_x\neq0$. Thus, the class is either $2r$ or $2r+1$. However, if the class was even, then we would have existence of a Liouville vector $\Delta_x$ and 
    $$0=\inn{\Delta_x} (\omega_x^{r+1}) =(r+1)\tau_x \wedge \omega_x ^r\,,$$
    which is a contradiction, so the class must be $2r+1$.
\end{proof}

\begin{proposition}\label{conditions_odd_class}
   Let $(\tau,\omega)$ be a pair of a differential 1-form and a 2-form.  
Then $(\tau,\omega)$ is a doublet of class $2r+1$ if, and only if,
    \begin{equation}\label{eq:odd_conditions}
            \omega ^{r+1} = 0\,,
            \qquad
            \tau \wedge \omega^r \neq 0\,.
    \end{equation}
\end{proposition}
\begin{proof}
    The result follows directly from the previous lemma.
\end{proof}


\begin{lemma}\label{lem:lepage}
    Let $(\tau,\omega)$ be a pair of a differential 1-form and a 2-form and $x\in M$ a point.
Then, $\cl (\tau_x,\omega_x) =2r$ if, and only if,
\begin{equation}
    \omega_x^{r}\neq 0\,, 
            \qquad
            \omega_x^{r+1}=0\,,\qquad 
            \tau_x \wedge \omega_x^r =0\,.
    \end{equation}
If $\tau_x\neq 0$, then $\omega_x^{r}\neq 0$ and $\tau_x \wedge \omega_x^r =0$ imply that $\omega_x^{r+1}=0$.
\end{lemma}
\begin{proof}
Suppose that $\cl(\tau_x,\omega_x)=2r$. 
Then, it follows from Proposition \ref{prop:doublet_conditions} that 
$\rk \omega_x=2r$. 
Hence, by Theorem~\ref{thm:even_rank}, we have
    $$\omega_x^{r}\neq 0, \qquad
            \omega_x ^{r+1} = 0\,.$$
Now, as the class is even, there exists a Liouville vector $\Delta_x$, and we have
    $$0=\inn{\Delta_x} (\omega_x^{r+1})= (r+1)\tau_x\wedge\omega_x^r\,,$$
which implies that $\tau_x\wedge \omega_x^r=0$.

For the converse,
first note that if $\omega^r_x\neq0$ and $\tau_x \wedge \omega^r_x=0$ then the class must be even. Indeed, if the class was odd, there would exist a Reeb vector $R_x$ such that
$$
0=\inn{R_x} (\tau_x \wedge \omega_x ^r)=\omega_x^r\,,
$$
which is a contradiction with $\omega_x^r\neq0$. 

If we, additionally, assume that $\tau_x\neq 0$, then from the fact there exists a Liouville vector $\Delta_x$, and
$$
\inn{\Delta_x} (\omega_x^{r+1})= (r+1)\tau_x\wedge\omega_x^r=0\,,
$$
we have, by Proposition~\ref{rank_powers},
that necessarily $\omega_x^{r+1}=0$, as $\Delta_x \notin \Ker \omega_x$.

The conditions $\omega_x^r\neq 0$ and $\omega_x^{r+1}=0$ imply that $\omega_x$ has rank $2r$, by Theorem~\ref{thm:even_rank}, and so the class must also be $2r$.

\end{proof}

\begin{remark}
  The last result in the preceding lemma is also known as Lepage's theorem, see \cite[Prop.\,2.5]{LM_1987}.  
\end{remark}

\begin{remark}
    Observe that if $\cl(\tau_x,\omega_x)=2r$ and $\tau_x\neq 0$, then as a direct consequence of Proposition~\ref{rank_powers} and the existence of a Liouville vector field, we have
    $$\tau_x \wedge \omega_x ^i \neq 0\,,$$
    for all $i=1,\dotsc, r-1$.
\end{remark}

\begin{proposition}\label{prop:even_conditions}
  
    Let $(\tau,\omega)$ be a pair of a differential 1-form and a 2-form. The pair $(\tau,\omega)$ is a doublet of class $2r$ if, and only if,
    \begin{equation}\label{eq:even_conditions}
\omega^r \neq 0\,, \qquad \omega ^{r+1}=0\,, \qquad \tau\wedge \omega^r=0\,.
    \end{equation}

    If, additionally, $\tau$ is nowhere-vanishing, then
    $(\tau,\omega)$ is a doublet of class $2r$ if, and only if,
    $$\omega^r \neq 0\,, \qquad \tau\wedge \omega^r=0\,.$$
    
\end{proposition}
\begin{proof}
    This result follows directly from the previous lemma.
\end{proof}


We can combine the two preceding propositions to obtain 
a relation between the class of a pair and 
some of their algebraic properties:

\begin{corollary}\label{cor:prepair_rank}
Let $(\tau, \omega)$ be a pair of a 1-form and a 2-form. 
The following statements are equivalent:
\begin{enumerate}[$(i)$]
    \item $\tau$ is nowhere-vanishing and $(\tau,\omega)$ has constant class, either $2r+1$ or $2r+2$.
    \item $\omega$ has constant rank,
    $\tau\wedge \omega^r \neq 0$, and $\tau \wedge \omega^{r+1} = 0$.
\end{enumerate}
    Provided the previous equivalent conditions hold, 
    if $\omega^{r+1} = 0$ then the class is odd, 
    and if $\omega^{r+1}\neq 0$ then the class is even.
\end{corollary}

\section{Geometrical structures associated with a doublet}
\label{ch_3}

In this section we introduce some geometrical objects that are associated with a doublet. Namely, we present and study the characteristic tensor, the extended characteristic distribution and the extended 3-form, and relate them to the class of the doublet and the Reeb or Liouville vector fields.

\subsection{The characteristic tensor field}

In the following a 2-covariant tensor field $B$ constructed on the basis of a pair $(\tau,\omega)$ will prove to be especially useful.
We introduce it as a sum of its symmetric and skew-symmetric parts:

\begin{definition}
The \dfn{characteristic tensor field} associated
with a pair $(\tau,\omega)$
is the 2-covariant tensor field given by
\begin{equation}
B = \tau \otimes \tau + \omega 
\,.
\end{equation}
\end{definition}

In general $B$ is not symmetric nor skew-symmetric.
This tensor field defines a vector bundle morphism,
$$
\begin{array}{rccl}
\widehat B \colon & \Tan M & \longto & \Tan^*M
\\ 
& v & \longmapsto & B(v,\cdot) = \inn{v}B = (\inn{v}\tau) \tau + \inn{v}\omega\,.
\end{array}
$$
So we have also a $\Cinfty(M)$-module morphism
that we can simply write as
$$
\begin{array}{rccl}
\widehat{B}\colon & \mathfrak{X}(M) & \longto & \Omega^1(M) \\
& X & \longmapsto & B(X,\cdot) = \inn{X}B =(\inn{X} \tau) \tau  +\inn{X} \omega \,.
\end{array}
$$

\begin{proposition}\label{prop:kernel_B}
Let $(\tau,\omega)$ be a pair of a differential 1-form and a 2-form. 
Then,
    $$\K_{(\tau,\omega)}=\Ker\tau \cap \Ker \omega = \Ker \widehat B = (\Im \widehat B)^{\circ}\,,$$
    holds at every point.
Thus,
$$\K_{(\tau,\omega)}^{\circ} =( \Ker \widehat B)^{\circ}=\Im \widehat B \,,$$
and, hence, the class of the pair is equal to the dimension of $\Im \widehat B$.
\end{proposition}
\begin{proof}
    Let us first see that $\Ker \tau \cap \Ker \omega =\Ker \widehat B$. The inclusion $\Ker \tau \cap \Ker \omega \subseteq\Ker \widehat B$ is trivial. For the other inclusion, assume that we have $\widehat B(X)=0$, i.e.\ 
    $$\widehat B(X)=(\inn{X} \tau)\tau +\inn{X}\omega = 0\,.$$
    Then, contracting again the expression by $X$, we obtain $(\inn{X} \tau)^2 = 0$, which implies that $\inn{X} \tau = 0$ and, hence, that $\inn{X} \omega = 0$. 
    So $X\in\Ker \tau \cap \Ker \omega$, as we wanted to see.

    Finally, we want to prove that $(\Im \widehat B)^{\circ} = \Ker \widehat B$. First we prove that $(\Im \widehat B)^{\circ} \supseteq \Ker \widehat B=\K_{(\tau,\omega)}$. Indeed, for any vector $Y$ and any $X\in \K_{(\tau,\omega)}$, we have
    $$\inn{X} \widehat B(Y) = (\inn{Y}\tau) \inn{X}\tau + \inn{X}\inn{Y}\omega =-\inn{Y}\inn{X}\omega = 0\,.$$
    As at every point both subspaces have the same dimension, they are necessarily equal.
\end{proof}

\begin{corollary}\label{cor:B_iso}
Let $(\tau,\omega)$ be a pair.
The following conditions are equivalent:
\begin{enumerate}[$(i)$]
\item 
$\widehat B$ is a vector bundle isomorphism.
\item 
The characteristic distribution of the pair is zero.
\item 
The class of the pair equals the dimension of the manifold.
\end{enumerate}
\end{corollary}

\begin{proposition}
\label{reeb_structure}
    Let $(\tau,\omega)$ be a doublet of odd class. Then, a vector field $X$ is a Reeb vector field for the doublet if and only if $\widehat B(X) = \tau$. Hence, all Reeb vector fields are given by $R=R_0+\Gamma$, where $R_0$ is a particular Reeb vector field and $\Gamma\in\K_{(\tau,\omega)}$.
\end{proposition}
\begin{proof}
    If $R$ is a Reeb vector field, it follows directly from its definition that $\widehat B(R)=\tau$. Conversely, if for some vector field $X\in \X (M)$ one has
    $$\widehat B(X) =  (\inn{X}\tau)\tau+\inn{X}\omega =\tau\quad \Longrightarrow\quad \inn{X}\omega=(1-\inn{X}\tau)\tau\,,$$
    then, contracting the expression with a Reeb vector field we obtain that, necessarily, $\inn{X}\tau=1$. From this, it follows that $\inn{X}\omega=0$. Hence, $X$ is a Reeb vector field.
\end{proof}
\begin{proposition}\label{prop:liouville_structure}
    Let $(\tau,\omega)$ be a doublet of even class. Then, a vector field $X$ is a Liouville vector field if and only if $\widehat B(X) = \tau$. Hence, all Liouville vector fields are given by $\Delta=\Delta_0+\Gamma$, where $\Delta_0$ is a particular Liouville vector field and $\Gamma\in\K_{(\tau,\omega)}=\Ker\omega$.
\end{proposition}
\begin{proof}
    If $\Delta$ is a Liouville vector field then one can check directly that $\widehat B(\Delta)=\tau$, as $\Delta\in \Ker\tau$ by definition. Conversely, if for some vector field $X$, one has
    $$  (\inn{X}\tau)\tau+\inn{X}\omega=\tau\quad \Longrightarrow\quad \inn{X}\omega=(1-\inn{X}\tau)\tau\,.$$
     Contracting this last expression with a Liouville vector field, 
    $$\inn{\Delta}\inn{X}\omega = -\inn{X}\tau=0\,,$$
    and so, necessarily, $X\in\Ker\tau$. This implies that $\inn{X}\omega=\tau$, so $X$ is a Liouville vector field.
\end{proof}
In this way, both Reeb and Liouville vector fields are characterized as the elements of the preimage~$\widehat B^{-1}(\tau)$.

\subsection{The extended characteristic distribution}

\begin{definition}
Let $(\tau,\omega)$ be a pair of a differential 1-form and a 2-form. 
The \dfn{extended characteristic distribution}%
\footnote{ In~\cite{GG_23} they refer to this just as the characteristic distribution.} 
at $x \in M $ is
$$
\mathcal{X}_{(\tau,\omega)_x} = 
\{ v  \in \Ker (\tau_x) \mid 
\inn{v} \omega_x = a \, \tau_x\,,  \text{ for some } a\in \R \}\,.
$$
\end{definition}
Notice that by definition 
$\K_{(\tau,\omega)x} \subseteq \mathcal{X}_{(\tau,\omega)x}$,
and if $\tau_x=0$ then the equality takes place.

\begin{lemma}
\label{prop_extended_dist}
Let $(\tau,\omega)$ be a pair of a differential 1-form and a 2-form and $x\in M$.
\begin{itemize}
    \item If the class of the pair at $x$ is odd, then
    $$ \mathcal{X}_{(\tau,\omega)_x} =  \K_{(\tau,\omega)_x} \,. $$

    \item  If the class of the pair at $x$ is even, then either
    $$ \tau_x = 0\quad\text{and}\quad\mathcal{X}_{(\tau,\omega)_x} =\K_{(\tau,\omega)_x} \,,$$
    or
    $$ \tau_x \neq 0 \quad \text{and}\quad\mathcal{X}_{(\tau,\omega)_x} = \K_{(\tau,\omega)_x} \oplus \left\langle \Delta_x \right\rangle\,,$$
    where $\Delta_x$ is any Liouville vector at $x$.
\end{itemize}
\end{lemma}   
\begin{proof} 
If the class is odd note that if 
$v\in \mathcal{X}_{(\tau,\omega)_x}$ 
then this implies that $\inn{v}\omega_x=0$, 
as otherwise we would have existence of Liouville vectors.
    If the class is even then the result follows immediately from the definition of Liouville vector.
\end{proof}

\begin{proposition}\label{prop:extended_char}
    Let $(\tau,\omega)$ be a doublet with $\tau$ nowhere-vanishing. The class is odd if, and only if,
    $$ \mathcal{X}_{(\tau,\omega)} =  \K_{(\tau,\omega)} \,. $$
    The class is even if, and only if,
    $$\mathcal{X}_{(\tau,\omega)} = \K_{(\tau,\omega)} \oplus \left\langle \Delta_0 \right\rangle \,,$$
    where $\Delta_0$ is any Liouville vector field of the doublet.
\end{proposition}
\begin{proof}
    This follows directly from the preceding lemma. 
\end{proof}


\begin{corollary}
\label{cor:extended_corollary}
   A doublet $(\tau,\omega)$, with $\tau$ nowhere-vanishing, has class $2r+1$ or $2r+2$ if and only if the extended characteristic distribution $\mathcal{X}_{(\tau,\omega)}$ is a regular distribution of rank $m-2r-1$.
\end{corollary}

Notice the similarity of this statement to that of 
Proposition~\ref{prop:doublet_conditions}.
We also have the following result that parallels Proposition~\ref{prop:extended_char}.

\begin{proposition}\label{prop:ker_omega}
Let $(\tau,\omega)$ be a doublet.
 The class of $(\tau,\omega)$ is odd if, and only if,
    $$\Ker \omega =\K_{(\tau,\omega)} \oplus \left\langle R_0 \right\rangle  \,,$$
    where $R_0$ is any Reeb vector field of the pair~${(\tau,\omega)}$.
The class $(\tau,\omega)$ is even if, and only if,
$$\Ker \omega = \K_{(\tau,\omega)}\,.$$
\end{proposition}
\begin{proof}
For even class, it follows directly from Theorem~\ref{thm:class-parity}. 
If the class is odd, then it is clear that $\K_{(\tau,\omega)} \oplus \left\langle R_0 \right\rangle$ is contained in $\Ker \omega$ for any Reeb vector field $R_0$.
    Moreover, if the pair has class $2r+1$, the distribution $\K_{(\tau,\omega)} \oplus \left\langle R_0 \right\rangle$ has dimension $ (m-2r-1)+1= m-2r$ and, as the rank of $\omega$ is $2r$, $\dim(\Ker \omega)=m-2r$. 
As the dimensions match, they must be equal.
\end{proof}

\subsection{The extended 3-form}

In this section we introduce a differential 3-form 
associated with a doublet and state its relation to the extended characteristic distribution and the class of the given doublet. 

\begin{definition}
The \dfn{extended $3$-form associated with a doublet} $(\tau,\omega)$ 
is the 3-form $\tau \wedge \omega$.
\end{definition}

\begin{lemma}
Suppose $\tau_x \neq 0$.
Then 
$\tau_x \wedge \omega_x = 0$ if and only if $\cl(\tau_x,\omega_x) \leq 2$.
\end{lemma}
\begin{proof}
This is a consequence of Lemmas \ref{lem:odd_conditions} and \ref{lem:lepage}.
\end{proof}

In view of this, to avoid trivial cases,
in what follows we will consider only pairs $(\tau,\omega)$
with class~$\geq 3$,
and therefore manifolds of dimension~$\geq 3$.

\begin{proposition}
\label{kernel_3-form_extended_characteristic}
Let $(\tau,\omega)$ be a doublet such that $\tau$ is nowhere-vanishing, 
$\cl (\tau,\omega)\geq 3$. 
The kernel of the extended 3-form is the extended characteristic distribution: 
$$ 
\Ker (\tau \wedge \omega) = \mathcal{X}_{(\tau,\omega)}\,.
$$
\end{proposition}
\begin{proof}
First notice that 
$X \in \Ker (\tau \wedge \omega)$ if and only if 
$$
(\inn{X} \tau) \omega = \tau\wedge (\inn{X} \omega)\,.
$$
From the definition of the extended characteristic distribution,
it is clear that 
$X\in \mathcal{X}_{(\tau,\omega)}$ implies 
$X\in \Ker {(\tau\wedge\omega)}$. 
Conversely, let $X$ be a vector field such that 
$(\inn{X} \tau) \,\omega = \tau\wedge (\inn{X} \omega)$. 
Then, left-wedging with $\tau$ on both sides of this expression, we obtain
$$
(\inn{X}\tau) \,\tau\wedge \omega =0\,.
$$
    As $\tau\wedge \omega \neq 0$ by the previous lemma, $\inn{X} \tau = 0$, and so
    $\tau\wedge (\inn{X} \omega)=0$. Hence, $\inn{X}\omega = f\tau$ for some function $f$, so $X \in \mathcal{X}_{(\tau,\omega)}$.
\end{proof}


Given a $p$-form $\alpha \in \Omega^p(M)$, if the map 
$v_x\in \T_x M\mapsto \inn{v_x} \alpha_x\in \bigwedge^{p-1} \T^*_x M$ 
is injective for every $x\in M$, 
then the form is said to be an \emph{almost-multisymplectic form} \cite{CIL_99}. 
If, additionally, the form is closed then it is said to be \emph{multisymplectic}.
The previous result allows us to characterize, in terms of the class, which 3-forms constructed as the wedge product of a 1-form and a 2-form are almost-multisymplectic.

\begin{corollary}
\label{3-form_injective}
Let $(\tau,\omega)$ be a pair of a 1-form and a 2-form on~$M$. 
The following statements are equivalent.
\begin{enumerate}[$(i)$]
\item 
The extended 3-form $\tau \wedge \omega$ is almost-multisymplectic.
\item 
$\tau$ is nowhere-vanishing and
the extended characteristic distribution $\mathcal{X}_{(\tau,\omega)}$ is zero.
\item 
$M$ is odd-dimensional and 
the class of the pair $(\tau,\omega)$ is equal to this dimension.
\end{enumerate}
\end{corollary}
\begin{proof}
    The equivalence between the first two statements comes from the previous proposition.
    The equivalence between the second and third statements follows from Lemma~\ref{prop_extended_dist}, as it implies that necessarily the characteristic distribution must be zero, so the class must equal the dimension of the manifold, and if the class was even then the extended characteristic distribution would have dimension~1.
\end{proof}


\section{Some important examples}
\label{sec_examples}

Many relevant geometric structures fit into the framework of doublets
presented so far.
Here we review some of them,
paying especial attention to the properties of the associated geometric structures.

\subsection{Contact structures}
A \dfn{contact form} on a manifold $M$ of dimension $2n+1$ is a 1-form $\eta$ such that
\begin{itemize}
\item 
$\eta \wedge (\d\eta)^n$ is a volume form on $M$, i.e.\ it is nowhere-vanishing.
\end{itemize}    
According to Proposition~\ref{conditions_odd_class}, 
this means that
the pair $(\eta, \d\eta)$ is a doublet of class $2n+1$. 

Then the characteristic distribution
$\K_{(\eta,\d\eta)}= \Ker \eta \cap \Ker \d\eta$ is zero, 
and so is the extended characteristic distribution.
By Proposition \ref{prop:directsum}
we have the direct sum decompositions
$
\Tan^* M = (\Ker \eta)^{\circ} \oplus (\Ker \d\eta)^{\circ} 
$
and
$
\Tan M = \Ker \eta \oplus \Ker \d \eta 
$.
Also, by Proposition~\ref{prop:kernel_B}, 
the characteristic tensor field defines an isomorphism
$\widehat B \colon \Tan M \to \Tan^*M$,
given by 
$\widehat B(X) = \inn{X} \d\eta + (\inn{X} \eta )\eta$.  
Since the class is odd, there exists a Reeb vector field $R$, which is uniquely defined as $R = \widehat B^{-1}(\eta)$. 
Corollary~\ref{3-form_injective} implies that 
the extended 3-form $\eta \wedge \d \eta$ 
is an almost-multisymplectic form of degree 3 (as long as $\dim M \geq 3$).

\subsection{Cosymplectic and precosymplectic structures}

    Let $M$ be a manifold of dimension $2n+1$.
    A \dfn{cosymplectic structure} on~$M$ is a pair 
    $(\tau, \omega)$, where $\tau \in \Omega^1(M)$ and $\omega \in \Omega^2(M)$ satisfy that
    \begin{itemize}
        \item both forms are closed, namely $\d \tau = 0$ and $\d\omega = 0$, and
        \item $\tau \wedge \omega^n$ is a volume form on $M$, i.e.\ it is nowhere-vanishing.
    \end{itemize}

The second condition is equivalent to the pair $(\tau, \omega)$ being of class $2n+1$.
As before, we have direct sum decompositions 
$
\Tan^* M = (\Ker \tau)^{\circ} \oplus (\Ker \omega)^\circ
$
and
$\Tan M = \Ker \tau \oplus \Ker \omega$.
Again, the map $\widehat B(X) = \inn{X} \omega +(\inn{X}\tau) \tau $ is an isomorphism, and there exists a uniquely determined Reeb vector field $R = \widehat B^{-1}(\tau)$. 
As both $\tau$ and $\omega $ are closed, the 3-form $\tau \wedge \omega \in \Omega^3(M)$ is also closed, and hence it is a multisymplectic form (as long as $\dim M \geq 3$), by Corollary~\ref{3-form_injective}.

It is important to note that, while most of 
these properties are shared with contact structures, the two structures are distinct. Their differences are reflected for instance in their local expressions; specifically, they possess different Darboux normal forms. In the contact case, the existence of Darboux coordinates is a consequence of the relation between the 1-form and its exterior derivative, while for cosymplectic structures, it follows from the closedness conditions.

A more general geometric framework can be obtained by relaxing the second condition.
A \dfn{precosymplectic structure} on $M$ is a pair $(\tau,\omega)$ of a 1-form and a 2-form such that
\begin{itemize}
    \item 
    both forms are closed, namely $\d \tau = 0$ and $\d \omega = 0$, and
    \item 
    the class of $(\tau,\omega)$ is constant, and $\tau$ is nowhere-vanishing.
\end{itemize}

A direct consequence of the closedness of these forms is the involutivity of their associated distributions: the kernel of $\omega$, the characteristic distribution, and the extended characteristic distribution. 

In the literature, it is usual to impose the class of the precosymplectic structure to be odd 
in order to have certain Darboux coordinates or 
existence of Reeb vector fields, among other reasons 
\cite{CLM_1994, GLRR_2024}.
This condition can be expressed, for instance, by asking 
$\K_{(\tau,\omega)}$ to be strictly contained in $\Ker \omega$, by Theorem~\ref{thm:class-parity}. 
If $\omega$ has rank $2r$, then this is equivalent to
$$\omega^{r+1}=0 \,, \qquad \tau \wedge \omega^r \neq 0\,.$$
If the class is odd there are Reeb vector fields
(Theorem \ref{thm:class-parity}).
Also, in this case we have $\mathcal{X}_{(\tau,\omega)} = \K_{(\tau,\omega)}$.

When  
$ \Ker \omega \subseteq \Ker \tau$ the class is even. 
If $\rk \omega = 2r$, this is equivalent to
$$
\omega^r\neq 0\,, \qquad 
\tau\wedge \omega^r =0\,,
$$
by Proposition~\ref{prop:even_conditions}.
In this case there do not exist Reeb vector fields,
but Liouville vector fields
(Theorem \ref{thm:class-parity}).

\begin{example}
    Let $M= \R^4$ with coordinates $(x,y,z,t)$ and consider the following differential forms
    $$
    \tau = \d t\,,  \qquad 
    \omega_1 = \d x \wedge \d y + y \,\d y \wedge \d z \,,\qquad  
    \omega_2 = \d x \wedge \d t \,. 
    $$
    All forms are closed and nowhere-vanishing, and we have $\omega_1^2 = \omega _2^2 = 0 $, so both 2-forms have rank 2. As
    $$
    \tau \wedge \omega_1 \neq 0\,, \qquad 
    \tau\wedge \omega_2 =0\,,
    $$
    both pairs $(\tau,\omega_1)$ and $(\tau,\omega_2)$ are precosymplectic structures of, respectively, class 3 and 2. 
    The Reeb vector fields of the first pair are 
    $$R = \fracp{}{t} + \Gamma_1\,,$$
    with 
    $\Gamma_1 \in \K_{(\tau,\omega_1)} = \left\langle  y\dparder{}{x}+\dparder{}{z}\right\rangle$.
    The Liouville vector fields of the second pair are
    $$\Delta  = \fracp{}{x } + \Gamma_2\,,$$
    with $\Gamma_2 \in \K_{(\tau,\omega_2)}= \left\langle \dparder{}{y},\dparder{}{z}\right\rangle$
\end{example}


\subsection{Symplectic and presymplectic structures}

A \dfn{symplectic form} on a manifold $M$ is a 2-form $\omega\in \Omega^2(M)$ such that
\begin{itemize}
\item 
it is closed:
$\d \omega = 0$, and
\item 
it is non-degenerate.
\end{itemize}
The non-degeneracy is equivalent to $\Ker \omega$ being zero, 
and hence to the rank of $\omega$ being equal to the dimension of the manifold. 
It is also equivalent to 
$\omega^n$ being a volume form.
Any symplectic manifold has even dimension.

A symplectic form can be understood as a doublet $(0,\omega)$ 
(where $0$ denotes the identically vanishing 1-form). 
The class of this doublet is obviously $2n$,
which is equivalent by Proposition~\ref{prop:even_conditions} to 
being $\omega^n\neq 0$ a volume form.  
In this case, 
the morphism induced by the characteristic tensor is the isomorphism 
$\widehat B = \widehat \omega$. 
The unique Liouville vector field of this doublet is the zero vector field.

If $\omega$ is exact, that is, if there exists a 1-form $\theta$ such that $\omega = \d \theta$, 
then we can also consider the doublet $(\theta, \omega)$. 
The class of this doublet is still $2n$, because the characteristic distribution is trivial. 
In this case we have 
$\Tan^* M =(\Ker \omega)^{\circ} \supset (\Ker \theta)^{\circ}$.
By Proposition~\ref{prop:kernel_B}, 
the map $\widehat B(X) = \inn{X} \omega + (\inn{X} \theta)\theta$ 
is an isomorphism, which gives $\Delta=\widehat B^{-1}(\theta)$. 
Note that we also have $\widehat \omega^{-1} (\theta) = \Delta$. 
One can check that $\widehat B$ is an isomorphism by rewriting it as
$$\widehat B(X) = \inn{X}\omega +(\inn{X} \theta)\inn{\Delta}\omega = \inn{X + (\inn{X} \theta)\Delta} \omega\,,$$
and then checking that $X+(\inn{X}\theta)\Delta = 0$ if, and only if, $X=0$. By Proposition~\ref{prop:extended_char}, the extended characteristic distribution $\mathcal{X}_{(\theta,\omega)}$ is spanned by the Liouville vector field $\Delta \in \X(M)$. 
Recall that, by Proposition~\ref{kernel_3-form_extended_characteristic}, if $\dim M \geq 3$ and $\theta$ is nowhere-vanishing, then $\mathcal{X}_{(\theta,\omega)} = \Ker (\theta\wedge \omega)$.

When the non-degeneracy condition is relaxed, we obtain more general geometric structures.
A \dfn{presymplectic structure} on a manifold $M$ is a $2$-form $\omega$ such that
\begin{itemize}
    \item $\omega$ is closed, and
    \item the rank of $\omega$ is constant.
\end{itemize}
In this case, we can still consider the doublet $(0,\omega)$, 
whose class is the rank of $\omega$, and hence an even number. 
The Liouville vector fields of this pair are the sections of $\Ker \omega$.

In the case that $\omega$ is exact, with $\omega = \d \theta$, 
then the pair $(\theta,\omega)$ does not necessarily have constant class. 
Even when the class is constant, it might be even or odd.

\subsection{Hamiltonian systems}

A \dfn{Hamiltonian system} is a triple $(M,\omega,H)$, where $M$ is a manifold, $\omega\in \Omega^2(M)$ is a symplectic form and $H$ is a function.
Hamiltonian systems are of utmost importance in 
analytic mechanics, symplectic geometry and dynamical systems theory.
Nevertheless, in some applications one is lead to consider the case where
$\omega$ is only a presymplectic form,
so we have a \dfn{presymplectic Hamiltonian system}. 


In this context, it is usual to search for the Hamiltonian vector fields, that is, the vector fields satisfying
$$
\inn{X}\omega = \d H\,.
$$
Within the framework of the previous sections, 
the Hamiltonian vector fields are just 
the Liouville vector fields of the pair $(\d H,\omega)$. 
If $\omega$ is symplectic, the Hamiltonian vector field is the unique Liouville vector field of the pair. 

When $\omega$ is degenerate, Hamiltonian vector fields 
may not exist or be unique.
This problem is of great importance in mathematical physics,
and so procedures to deal with it have been developed,
starting with the Dirac--Bergmann theory of constraints.
Its geometric aspects have been studied by many authors 
(see, for instance,~\cite{Car_1990,GNH_1978}).

Let us further analyse this problem.
The first step to obtain a Hamiltonian vector field 
is to try to solve the linear equation
$$
\inn{X_x} \omega_x = \d_x H
\,,
$$
which has a solution if and only if
$$ 
\d_x H \in \Im \widehat\omega_x\,.
$$
According to Theorem \ref{thm:class-parity}, 
this is equivalent to saying that 
the pair $(\d H, \omega)$ has even class at~$x$.
Therefore we arrive at the following interpretation
in the language of constrained systems:
\emph{the primary constraint submanifold
is the set of points where the class of $(\d H, \omega)$ is even.}
This could be the starting point of a constraint algorithm,
but we will not delve into this.

\subsection{Locally conformally symplectic manifolds}

Sometimes one refers as an
\dfn{almost symplectic form} on $M$ 
a 2-form $\omega\in \Omega^2(M)$ 
that is just
\begin{itemize}
\item 
nondegenerate.
\end{itemize}
This means that the closedness condition to be symplectic 
has been dropped. 
Nevertheless it is more interesting the case where symplecticity can be achieved just by multiplying by a function~%
\cite{CG_2025}.
A $2$-form $\omega$ is said to be \dfn{conformally symplectic} 
if there exists a function $f$ such that 
$\mathrm{e}^f \omega$
is a symplectic form.
In other terms:
\begin{itemize}
\item 
$\d (e^f \omega) = 0$ and
\item 
$\omega$ is nondegenerate.
\end{itemize}
A 2-form is \dfn{locally conformally symplectic}
if this conformal factor can be found locally around each point.




A $2$-form $\omega$ is locally conformally symplectic if, and only if, there exists a \emph{closed} $1$-form $\theta$ 
(the so-called \dfn{Lee form}) 
such that
$$
\d \omega = - \theta \wedge \omega\,.
$$
The 2-form is globally conformally symplectic if moreover $\theta$ is exact.

Since $\omega^n\neq 0$ 
the pair $(\theta, \omega)$ is a doublet of class $2n$. 
This implies that there exists a unique Liouville vector field 
$\Delta$ such that 
$\inn{\Delta} \omega  = \theta$.
Again, in this case, both the map 
$\widehat B(X) = \inn{X} \omega + (\inn{X} \theta) \theta$ and 
$\widehat\omega(X) = \inn{X} \omega $ 
are isomorphisms, and 
$\Delta = \widehat\omega^{-1} (\theta)= \widehat B^{-1}(\theta)$.
Finally notice that $\d\omega$ is minus the extended 3-form of the doublet.


\section{Precontact forms}
\label{sec_precontact}



In this section, 
we systematically investigate the notion of precontact form.
A 1-form $\eta$ on a manifold defines a pair 
$(\eta,\d\eta)$,
to which the results presented in Sections~\ref{ch_2} and~\ref{ch_3}
can be applied.
In particular,
the characteristic distribution of the pair
$\K = \K_{(\eta, \d\eta)} =
\Ker \eta \cap \Ker \d\eta$
is the characteristic distribution of~$\eta$,
and its class 
$\cl(\eta, \d \eta)$
is the class of~$\eta$,
as we have noted in Remark~\ref{rmrk:class}.


\subsection{Precontact forms}
\label{precontact_structures}

\begin{definition}
A 1-form $\eta$ on a manifold is a \dfn{precontact form} if
\begin{itemize}
    \item 
    $\eta$ is nowhere-vanishing, and
    \item 
    the pair $(\eta, \d \eta)$ has constant class.
\end{itemize}
A \dfn{precontact manifold} is a pair $(M,\eta)$,
where $M$ is a manifold and $\eta$ is a precontact form on~$M$.
\end{definition}

Next, we apply the results from Sections~\ref{ch_2} and~\ref{ch_3} 
to obtain different characterizations of the notion of precontact form. In first place, rephrasing Corollary~\ref{cor:prepair_rank} we obtain the following.

\begin{proposition}
Let $\eta$ be a 1-form.
\begin{enumerate}[$(i)$]
\item 
$\eta$ is a precontact form of class $2r+1$ if, and only if, 
$$ 
            \eta\wedge(\d\eta)^{r} \neq 0\,,
            \qquad
            (\d\eta)^{r+1} = 0\,.
$$
\item 
$\eta$ is a precontact form of class $2r+2$ if, and only if,
$$  
            \eta \neq 0\,, \qquad \eta\wedge(\d\eta)^{r+1} = 0\,,
            \qquad
            (\d\eta)^{r+1} \neq 0\,.
    $$
\end{enumerate}
\end{proposition}




The definition presented here is slightly different from other definitions that can be found in the literature.
On one hand, it is more general than that of 
\cite{LGGMR_23,LL_19a,Lai_22}
as therein only precontact forms of odd class are considered.
On the other hand, 
in~\cite{GG_23},
precontact distributions are studied rather than precontact forms,
which nevertheless appear as 
nowhere-vanishing 1-forms satisfying
$$
\eta \wedge (\d \eta)^r \neq 0\,, \qquad \eta \wedge (\d \eta )^{r+1} = 0\,.
$$
Notice, however,
that this definition of precontact form 
is not enough 
to ensure that the rank of $\d\eta$ is constant, 
since the parity of the class could change at different points,
as it is showed in the next example.

\begin{example}
Consider the manifold $M=\mathbb{R}^3$ with coordinates $(x,y,s)$, endowed with the nowhere-vanishing $1$-form
$$
\eta = \d s - s x \,\d x 
.
$$
Straightforward computations give 
$$
\d\eta = x \,\d x \wedge \d s \,,
\qquad
\eta \wedge \d \eta = 0\,.
$$
Therefore:
\begin{itemize}
\item
Where $x=0$: $\d\eta$ has rank~0 and the class of $\eta$ is~1.
\item
Where $x\neq 0$: $\d\eta$ has rank~2 and the class of $\eta$ is~2.
\end{itemize}

\end{example}

There exists Darboux theorems characterizing the local expressions of precontact forms, we refer the reader to \cite{god_1969,LM_1987} for a proof.
\begin{theorem}[Darboux theorem, odd class]\,

\noindent Let $\eta \in \Omega^1(M)$ be a nowhere-vanishing $1$-form of constant class~$2r+1$. 
    Then, for all $x\in M$, there exist local coordinates 
    $q^1,\dotsc,q^r,p_1,\dotsc,p_r,s,u_1,\dotsc,u_z$ 
    (with $2r+z+1=m$) 
    such that
    \begin{equation}
    \label{pc_darboux_odd}
        \eta = \d s - \sum ^{r}_{i=1} p_i \,\d q^i \,.
    \end{equation}
\end{theorem}
    
In Darboux coordinates, 
the characteristic distribution and
the Reeb vector fields read
$$
\K = \bigg\langle \bigg\{ \fracp{}{u_a} \bigg\}_{a=1,\dotsc,z} \bigg\rangle\qquad\text{and}\qquad R = \fracp{}{s} +\Gamma\,,\quad \text{with }\ \Gamma \in \K\,.
$$

\begin{theorem}[Darboux theorem, even class]\,

\noindent Let $\eta \in \Omega^1(M)$ be a nowhere-vanishing $1$-form of constant class~$2r+2$. 
    Then, for all $x\in M$, there exists local coordinates $q^1,\dotsc,q^{r+1},p_1,\dotsc,p_{r+1},u_1,\dotsc,u_{
    z}$ (with $2r+2+z=m$) such that $p_1$ is nowhere-vanishing and
    $$
    \eta =\sum ^{r+1}_{i=1}p_i \d q^i \,.
    $$
\end{theorem}

In Darboux coordinates, the characteristic distribution and Liouville vector fields of $\eta$ read
$$ 
\K = \bigg\langle \bigg\{ \fracp{}{u_a} \bigg\}_{a=1,\dotsc,z} \bigg\rangle\qquad\text{and}\qquad \Delta = p_i\parder{}{p_i} + \Gamma\,,\quad  \text{with }\ \Gamma\in\K\,.
$$

Next we study the extended characteristic distribution.

\begin{proposition}\label{prop:preco_involutivity}
    If $(M,\eta)$ is a precontact manifold, then the kernel of $\d \eta$, the characteristic distribution, and the extended characteristic distribution, are involutive.
\end{proposition}
\begin{proof}
The kernel of $\d \eta$ is involutive because $\d \eta$ is a closed form.
    
For the characteristic distribution we have that if $X,Y \in \K = \Ker \eta \cap \Ker \d \eta$ then 
$$\inn{[X,Y]}\eta = \Lie_X (\inn{Y} \eta) - \inn{Y} \Lie _X \eta = -\inn{Y}\inn{X}\d\eta =0\,, 
\qquad
\inn{[X,Y]}\d \eta = \Lie_X (\inn{Y} \d \eta) - \inn{Y} \Lie _X \d \eta =  0\,, $$
which proves that it is involutive.

Last, for the extended characteristic distribution, let us discuss the two possible cases. If the class is odd, then the extended characteristic distribution is equal to the characteristic distribution. If the class is even, the extended characteristic distribution is $\mathcal{X}= \K \oplus \langle \Delta_0\rangle$, where $\Delta_0$ is any Liouville vector field. Thus, we only need to check that $[\Delta_0 ,\Gamma] \in \mathcal{X} $ for any $\Gamma \in \K$. We have
$$\inn{[\Delta_0, \Gamma]} \d\eta = -\inn{\Gamma} \d\eta = 0\,,$$
and as $\Ker \d \eta \subset \Ker \eta$ in this case, we conclude that the distribution is involutive. 
\end{proof}

\begin{remark}\label{rem:Liou_involutivity}
In the preceding proof we have shown indeed that,
for a precontact manifold of even class, 
$[\Delta,\Gamma]\subset \K$ for any Liouville vector field $\Delta$ and $\Gamma\in \K$.
\end{remark}

If $(M,\eta)$ is a precontact manifold with class $\geq 3$, then $\Ker (\eta \wedge \d \eta)$ is also involutive, as $\Ker (\eta \wedge \d \eta) = \mathcal{X}_{(\eta,\d \eta)}$.

\begin{example}
Here we present an example of a precontact structure of even class 
coming from the study of singular Lagrangian functions within the framework of contact mechanics~\cite{GMRR_2025}.

Consider the manifold $M = \Tan^*(\R ^3) \times \R$ with coordinates $(x,y,z;p_x,p_y,p_z;s)$ and the $1$-form
$$\eta = \d s - p_x \d x + \gamma s\d y - p_z\d z \,,$$
where $\gamma \in \R$ is a non-zero constant. We have
$$\d \eta = \d x \wedge \d p_x + \gamma \d s \wedge \d y + \d z \wedge \d p_z\,,$$
and one can easily check that 
$$\Ker \d \eta = \left\langle \fracp{}{p_y}\right\rangle \subset \Ker \eta\,,$$
and so 
$\K = \Ker \eta \cap \Ker \d \eta = \big\langle\fracp{}{p_y} \big\rangle$. This implies that the class is 6. As the class is even, there exist Liouville vector fields. Indeed, they are given by
$$\Delta = -\frac{1}{\gamma} \fracp{}{y} + p_x \fracp{}{p_x} + f\fracp{}{p_y} + p_z \fracp{}{p_z} + s\fracp{}{s}\,,$$
where $f$ is any arbitrary function.

\end{example}

\subsection{Contact and precontact Hamiltonian systems}\label{precontact_Hamiltonian}

A \dfn{contact Hamiltonian system} is a triple $(M,\eta,H)$
where $(M,\eta)$ is a contact manifold and 
$H \in \Cinfty(M)$ is the \dfn{Hamiltonian function} of the system.
This yields a dynamics on~$M$, given by the
\dfn{contact Hamiltonian vector field} $X_H \in \X(M)$,
which is defined by any of the following equivalent conditions:
\begin{enumerate}[(1)]
\item 
$\inn{X_H}\d\eta = \d H - (\Lie_R H)\eta$ \quad and \quad $\inn{X_H}\eta = -H$,
\item $\Lie_{X_H}\eta = -(\Lie_R H)\eta$ \quad and \quad $ \inn{X_H}\eta = -H$,
\item $\widehat B(X_H) = \d H - (\Lie_R H + H)\eta$, 
\end{enumerate}
where $R \in \X(M)$ is the (unique) Reeb vector field defined by the contact form.

It is clear that these formulations are not well-suited to define Hamiltonian dynamics on precontact manifolds, as Reeb vector fields may not exist nor be unique.
In a recent paper~\cite{GMRR_2025}, we obtained two equivalent formulations of the contact Hamiltonian equations which can be written without Reeb vector fields:
a vector field $X$ on~$M$ is the contact Hamiltonian vector field for $H$ if, and only if,
\begin{equation}\label{cH_eq_1}
    (\inn{X} \d \eta) \wedge \eta = (\d H) \wedge \eta 
				\qquad \text{and} \qquad
				\inn{X} \eta = -H \,,
\end{equation}
or, equivalently,
\begin{equation}\label{cH_eq_2}
    \inn{X} (\eta \wedge \d \eta) = - H \d \eta + \d H \wedge \eta\,.
\end{equation} 
\begin{remark}
    Equations \eqref{cH_eq_1} and \eqref{cH_eq_2} are equivalent if $\eta \wedge \d \eta \neq 0$. In other words, they are equivalent if $\eta$ is nowhere-vanishing and the class of $\eta$ is $\geq 3$. Otherwise, we also need to add the condition $\inn{X} \eta = -H$ to equation~\eqref{cH_eq_2}.  
\end{remark}
    
We want to extend these notions to the case of precontact forms.

\begin{definition}
    Given a precontact form $\eta$, and a function $H\colon M \to \R$, we say that the triple $(M,\eta,H)$ is a \dfn{precontact Hamiltonian system}.
\end{definition}

Given a precontact Hamiltonian system $(M,\eta,H)$, 
a \dfn{precontact Hamiltonian vector field} 
is a vector field $X \in \X (M)$ that satisfies equations~\eqref{cH_eq_1}.


Note that, if $\eta$ has odd class, and thus there exist Reeb vector fields, 
contracting the first condition of equations~\eqref{cH_eq_1} by any Reeb vector field $R\in \X (M)$ one obtains
\begin{equation}
	\inn{X} \d \eta = \d H - (\Lie_R H)\eta \,,
\end{equation}
which is precisely the usual way to define the contact Hamiltonian vector field
(condition~(1) at the beginning of this section).
As the Reeb vector field is not necessarily unique in the precontact case, 
this also implies that for there to be a solution $X_x$ at a point $x\in M$ 
it is necessary that
\begin{equation}
    \Lie_R H(x) = \Lie_{R'} H(x)\,,
\end{equation}
for any two Reeb vector fields $R,R'$.
By Proposition \ref{reeb_structure}, this is equivalent to
\begin{equation} \label{odd_constraint}
    \Lie_\Gamma H (x)=0\,,
\end{equation}
for every $\Gamma \in \mathcal{K}=\Ker \eta \cap \Ker \d \eta$.

Similarly, if $\eta$ has even class,
contracting the first equation of~\eqref{cH_eq_1} by any Liouville vector field $\Delta \in \X(M)$, 
we obtain the necessary condition
\begin{equation}\label{even_constraints}
  \Lie_{\Delta} H (x) = H(x)\,.  
\end{equation}
Again, this implies that $\Lie_\Gamma H(x) = 0$ for any $\Gamma \in \mathcal{K}= \Ker \d \eta$.

Last, let us point out the following result, which motivates the next section of this article. 
Recall that the codimension 1 distribution spanned by kernel of $\eta$ is the same as 
that spanned by the kernel of $g\eta$, 
for any nowhere-vanishing function~$g\in\Cinfty(M)$.
Also, it is well-known that the contact Hamilton equations of $(M,\eta,H)$ 
are the same ones as those defined by the contact Hamiltonian system $(M,g\eta,g H)$, 
with $g\in\Cinfty(M)$ any nowhere-vanishing function
\cite{sandwich,inverse-contact}.
For precontact Hamiltonian systems we have the following result.
\begin{proposition}
    Let $(M,\eta, H)$ be a precontact Hamiltonian system. The equations~\eqref{cH_eq_1} defined by the precontact system $(M,g\eta,\widetilde H)$, with $g\in\Cinfty(M)$ nowhere-vanishing and $\widetilde H\in\Cinfty(M)$, are the same as those defined by $(M,\eta,H)$ if and only if $\widetilde H = g H$.   
\end{proposition}

\begin{proof}
    First assume that $\widetilde H = gH$, with $g$ nowhere-vanishing. Then, the precontact Hamilton equations for $(M,g\eta, gH)$ are
\begin{equation}
		(\inn{X} \d (g\eta)) \wedge g\eta = \d (gH) \wedge (g\eta) \,,
		\qquad
		\inn{X}( g\eta )= -gH \,.
\end{equation}
The second equation, using that $g$ is nowhere-vanishing, implies that
 $\inn{X}\eta = -H$. It only remains to see that the first condition of~\eqref{cH_eq_1} is also satisfied. For this, note that
$$(\inn{X} \d (g\eta)) \wedge g\eta = (\d (gH)) \wedge (g\eta) \quad \Longrightarrow \quad ((\inn{X}\d g)  \eta -(\inn{X}\eta)\d g +g \inn{X}\d \eta )\wedge g\eta = (H \d g + g\d H)\wedge g \eta$$
and, using that $\inn{X}\eta = -H$, we obtain
$$(gH)\d g \wedge \eta + g^2(\inn{X}\d \eta\wedge \eta) = (gH)\d g \wedge \eta + g^2(\d H \wedge \eta)\,,$$
which, as $g$ is nowhere-vanishing, implies that
$\inn{X}(\d\eta) \wedge \eta = \d H \wedge \eta$,
as we wanted to see.

Conversely, assume that $(M,g\eta, \widetilde H)$ defines the same precontact Hamilton equations as $(M,\eta,H)$. Then, for $(M,g\eta,\widetilde H)$, we have
$\inn{X}(g\eta ) = -\widetilde H$.
Using that $g$ is nowhere-vanishing
$\displaystyle\inn{X}\eta = -\widetilde H/g\,.$
However, since both systems define the same contact Hamilton equations, 
we also have $\inn{X}\eta = -H$. 
Therefore, $\widetilde H = gH$, as we wanted to see.
\end{proof}

This proposition motivates us to study the effects of a conformal change of a precontact form~$\eta$, and in particular how it affects the class.
This is the subject of the next section.

\section{Conformal changes of parity}\label{sec_changes}

Let $\eta$ be a precontact form on a manifold~$M$.  
If we rescale it by a nowhere-vanishing function 
$f\in \Cinfty(M)$ we still have
$
\Ker\eta = \Ker(f\eta)\,.
$
Nevertheless, the resulting precontact structure $(M,f\eta)$ may behave in a
substantially different way from the original one.  
Some geometric structures associated with $\eta$, such as the Reeb and
Liouville vector fields, depend on $\eta$ in a non-trivial manner and
are therefore sensitive to such conformal transformations.
In fact, even the parity of the class of the precontact form can change,
as we are going to show.

Indeed, analysing the properties of the precontact form $f\eta$ amounts to understanding the pair
$$(f \eta, \d(f\eta)) = (f\eta ,  \d f \wedge \eta +f\d\eta)\,, $$
and it is clear that the additional term $\d f\wedge\eta$ can significantly modify the
underlying structure.

Instead of multiplying by a nowhere-vanishing function, 
it will be more convenient to write this multiplier as an exponential
$e^f$ of a function.
Also, to gain a wider perspective, we will consider a doublet $(\tau,\omega)$ of forms, with $\tau$ nowhere-vanishing, and, inspired by the previous comments, 
we will study the pair 
$$
\left(\strut e^f\tau, e^f(\d f\wedge \tau +\omega) \right) \,.
$$

\begin{remark}
Given a pair $(\tau,\omega)$, 
one could consider a simpler conformal transformation of the type
$(g\tau, g\omega)$,
with $g$ nowhere-vanishing.
If $(\tau,\omega)$ has constant class, $(g\tau, g\omega)$ has the same constant class, 
as their characteristic distributions coincide. 
The Reeb or Liouville vector fields of both pairs are also trivially related: 
if $\Delta$ is a Liouville of $(\tau,\omega)$ then it is also a Liouville of $(g\tau,g\omega)$, 
and if $R$ is a Reeb of $(\tau,\omega)$ then $R/g$ is a Reeb of $(g\tau,g\omega)$. 
However, these are not the conformal transformations natural to precontact structures.
\end{remark}

\begin{lemma} 
\label{lemma_formula_ef_times_form}
Let $\tau\in\Omega^1(M)$ be a $1$-form on $M$ and $f\in \Cinfty(M)$ a function. 
Then, for every $n \geq 1$, 
$$ (e^f(\d f\wedge \tau +\omega))^n = e^{nf}(n\, \d f \wedge\tau\wedge\omega^{n-1}+\omega^n)\,.$$
\end{lemma}
\begin{proof}
    Let us prove this by induction. For the case $n=1$ it is trivial.
    Now, assuming it is true for $n-1$, for the induction step we have
    \begin{align}
        (e^f(\d f\wedge \tau +\omega))^n &= (e^f(\d f\wedge \tau +\omega))^{n-1}\wedge(e^f(\d f\wedge \tau +\omega))
        \\
        &= e^{(n-1)f}((n-1)\d f\wedge \tau \wedge \omega^{n-2} + \omega^{n-1}) \wedge e^f(\d f\wedge\tau + \omega)
        \\
        &=e^{nf}(n\, \d f \wedge\tau\wedge\omega^{n-1}+\omega^n)\,,
    \end{align}
    which proves the result.
\end{proof}

Using this lemma we have
\begin{equation}\label{eq:after_lemma_eq}
    (e^f\tau) \wedge (e^f(\d f\wedge \tau +\omega))^n = e^{(n+1)f}(\tau \wedge \omega^n)\,,
\end{equation}
for all~$n$.

\begin{remark}
By Corollary~\ref{cor:prepair_rank}, 
if a pair $(\tau,\omega)$ has class $2r+1$ or $2r+2$ at a point, 
then the pair $(e^f \tau , e^f(\d f \wedge\tau + \omega))$ also has class $2r+1$ or $2r+2$ at the same point.
This class can either stay the same or change by one, i.e., from $2r+1$ to $2r+2$ or from $2r+2$ to $2r+1$. 
Moreover, even if $(\tau,\omega)$ has constant class,
its conformal rescaling may not, 
since the rank of $e^f(\d f \wedge\tau + \omega)$ may not be constant,
as shown in the next example.
\end{remark}

\begin{example} 
Consider $M = \R ^4$ with coordinates $(x,y,z,t)$, and 
$$\omega = \d x \wedge \d y\,, \qquad \tau = \d z\,, \qquad f= \frac{1}{2}t^2\,.$$
Then, $e^f(\d f \wedge \tau + \omega) = e^f (t\d t\wedge \d z +\d x \wedge \d y)$ and is clear that the rank of this 2-form is 4 if $t\neq 0$ and $2$ if $t=0$. 
It is clear that the class of the pair $(e^f\tau, e^f(\d f \wedge \tau + \omega))$ is $4$ if $t\neq0 $ and 3 if $t=0$. 
\end{example}

Given a precontact form $\eta$ on $M$, we can consider the form $z\eta$ on $M\times \R^{\times}$, where~$z$ is the canonical coordinate of the multiplicative group of nonzero reals $\R^{\times}$, and where we are considering the pull-back of $\eta \in \Omega^1(M)$ to $M\times \R^{\times}$. This process is known as the \dfn{presymplectization} of the precontact form (see~\cite{GG_23} for more details). We have the following result:
\begin{proposition}
Let $\eta$ be a precontact form on $M$ of class $2r+1$ or $2r+2$.
    Then the form $z\eta\in \Omega^1(M\times \R^{\times})$ is a precontact form of class $2r+2$.
\end{proposition}
\begin{proof}
To prove this we can use some of the results we presented before. When we pull-back the $1$-form $\eta$ to $M\times \R^\times$ the rank of the characteristic distribution increases by 1, as $\fracp{}{z}$ is in the kernel of the pull-backed $1$-form and its differential. Thus, the codimension of the characteristic distribution, i.e.\ the class, remains the same. The $1$-form $z\eta \in \Omega^1(M \times \R^{\times})$ is necessarily of even class, because the vector field $z\fracp{}{z}$ is a Liouville vector field. Indeed, we have
$$\inn{z\fracp{}{z}}\d (z\eta) = \inn{z\fracp{}{z}}(\d z \wedge \eta + z\d \eta) = z\eta\,.$$
Note that the vector field $z\fracp{}{z}$ is also the fundamental vector field of the $\R ^{\times}$-action and $1\in \R$, naturally defined on $M \times \R^{\times}$.
    
\end{proof}

\subsection{Change of even parity}

The next propositions use the previous results to characterize which functions alter the parity of the class of a doublet $(\tau,\omega)$ of even class.

\begin{proposition}\label{change_of_class_parity}
    Let $(\tau,\omega)$ be a doublet of class $2r+2$, with $\tau$ nowhere-vanishing, and $f\in\Cinfty(M)$ a function. Then, $(e^f\tau, e^f(\d f\wedge \tau+\omega))$ is a doublet of class $2r+1$ if, and only if, for all Liouville vector fields, we have $$\Lie _{\Delta} f=-1\,.$$
\end{proposition}
\begin{proof}
   Let us first assume that 
   $(e^f\tau, e^f(\d f\wedge \tau+\omega))$ 
   is a doublet of class $2r+1$. We will denote $\widetilde \omega \coloneqq \d f \wedge \tau+\omega$. As the doublet $(\tau,\omega)$ has even class, there must exist Liouville vector fields $\Delta \in \X(M)$.
    If we contract $(e^f\widetilde \omega)^{r+1}$ by any Liouville vector field $\Delta$ of $(\tau,\omega)$, we obtain
\begin{align}\label{Liouville_function_class}
    \begin{aligned}
        \inn{\Delta}(e^f\widetilde \omega )^{r+1} &=  e^{(r+1)f}((r+1)\, \inn{\Delta}(\d f \wedge\tau\wedge\omega^{r})+\inn{\Delta}\omega^{r+1})
        \\
        &=e^{(r+1)f}(((r+1)\Lie_{\Delta}f)\,\tau\wedge\omega^{r} + (r+1) \,\tau\wedge\omega^{r})
        \\
        &=(e^{(r+1)f}(r+1)(\Lie_{\Delta} f +1 ))\,\tau\wedge\omega^{r}\,.
        \end{aligned}
    \end{align}
    where we have used Lemma \ref{lemma_formula_ef_times_form} in the first equality.
Now, as the class of $(\tau,\omega)$ is $2r+2$, and the class of $(e^f \tau, e^f \widetilde \omega )$ is $2r+1$, by Propositions~\ref{conditions_odd_class} and~\ref{eq:even_conditions}, we have
$(e^f\widetilde \omega)^{r+1}=0$ and $ \qquad\tau\wedge\omega^r\neq 0$. 
Therefore, it follows from equation \eqref{Liouville_function_class} that 
$ \Lie_{\Delta} f +1=0 $,
for all Liouville vector fields $\Delta$ of the doublet $(\tau,\omega)$, as we wanted to see.

To see the converse first recall that all the Liouville vector fields $\Delta\in\X(M)$ of $(\tau,\omega)$
can be written as $\Delta = \Delta_0+\Gamma$, where $\Delta_0\in\X(M)$ is a particular Liouville and $\Gamma\in \Ker\omega$ (see Proposition~\ref{prop:liouville_structure}). Thus, if the assumption is satisfied for every Liouville vector field, one has that
$$\inn{\Delta_0}\d f+1=0\qquad \text{and} \qquad \inn{\Delta_0+\Gamma}\d f +1=\inn{\Delta_0}\d f+\inn{\Gamma}\d f+1=0\,,$$
which implies that 
$\inn{\Gamma}\d f = 0$,
for all $\Gamma \in \Ker\omega$. 
Using this and $\Ker\omega \subset \Ker \tau$ (due to the class being even), it is easy to check that
$\inn{\Gamma} (e^f\widetilde \omega) = 0$,
and also that
$ \inn{\Delta}(e^f\widetilde \omega) = e^f((\inn{\Delta} \d f)\tau + \tau) = 0$.
Thus, we have
$ \mathcal{X}_{(\tau,\omega)} =\K_{(\tau,\omega)} \oplus \left\langle \Delta_0 \right\rangle  \subseteq \Ker (e^f \tau )\cap \Ker (e^f\widetilde \omega)$,
but as the class of the modified pair cannot increase or decrease by more than 1, we necessarily have
$$\mathcal{X}_{(\tau,\omega)}=\K_{(\tau,\omega)} \oplus \left\langle \Delta_0 \right\rangle  = \Ker (e^f \tau )\cap \Ker (e^f\widetilde \omega)\,,$$
which is the characteristic distribution of the conformally transformed pair
$(e^f\tau,e^f\widetilde \omega)$;
therefore its class is $2r+1$.
\end{proof}


\begin{example}
    A function that changes the parity of a pair $(\tau,\omega)$ of even class may not always exist. 
    Consider $M = \R ^4$, $\tau= \d z$ and $\omega =  \d t\wedge \d z - y\d x\wedge \d z + x\d y\wedge \d z$. Then
    $$\Ker \tau= \left\langle \fracp{}{x} , \fracp{}{y},\fracp{}{t} \right \rangle\,,\qquad \Ker \omega= \left\langle \fracp{}{x} +y \fracp{}{t}, \fracp{}{y}-x\fracp{}{t} \right \rangle\,.$$
    As $\Ker \omega \subset \Ker \tau$, by Theorem~\ref{thm:class-parity}, the class is even, and equal to $2$. 
    Any Liouville vector field of the pair is given by
    $\Delta = \fracp{}{t} + \Gamma\,,$
    where $\Gamma \in \Ker \omega$. 
    Thus, to find a function $f \in \Cinfty(M)$ satisfying $\Lie_\Delta{f} = -1$ for all Liouville vector fields is equivalent to solving the system of linear PDE's
    $$ 
    \fracp{f}{t} = -1 \,,\qquad 
    \fracp{f}{x}+y\fracp{f}{t} = 0\,, \qquad 
    \fracp{f}{y}-x\fracp{f}{t} = 0\,,  
    $$
    which has no solution.
\end{example}

The next result provides sufficient conditions for the existence of a function that changes the parity.

\begin{proposition}\label{sufficient_cond_function}
    Let $(\tau,\omega)$ be a doublet of class $2r+2$, with $\tau$ nowhere-vanishing. Suppose that:
    \begin{itemize}
        \item $\K_{(\tau,\omega)} = \Ker \omega$ is an integrable distribution, and
        \item for every Liouville vector field~$\Delta$ of $(\tau,\omega)$ and $\Gamma \in \Ker \omega$, we have $[\Delta,\Gamma] \subset \Ker \omega$.
    \end{itemize}
    Then, around any point there exists a neighbourhood $V$ and a nowhere-vanishing function $g\in \Cinfty(V)$ 
    such that the pair $(g\tau,\d g\wedge \tau + g\omega)$ has constant class $2r+1$ on $V$.
\end{proposition}
\begin{proof}
    Let $M$ be a manifold of dimension $m$. By hypothesis, $\Ker \omega$ defines an integrable distribution of dimension $\mu = m-2r-2$. 
    By Frobenius' theorem, at any point $x \in  M$, there exist local coordinates $(x^1, \dotsc, x^{\mu},x^{\mu+1}, \dots, x^m)$ such that $\Ker \omega = \langle \fracp{}{x^1},\dotsc, \fracp{}{x^{\mu}}  \rangle$. 
    
    Let $\Delta_0$ be a particular Liouville vector field of the  doublet $(\tau,\omega)$. The condition 
    $[\Delta_0 , \Gamma] \in \Ker \omega$
    for all $\Gamma \in \Ker \omega$, implies that
    $$\Delta_ 0 = \sum_{i=1}^\mu a_i(x) \fracp{}{x^i} + \sum_{j=\mu+1}^{m}b_j(x^{\mu+1},\dotsc,x^m)\fracp{}{x^j}\,.$$
    By Proposition~\ref{change_of_class_parity}, we need to find a function $f$ such that $\Lie_\Delta f = -1$ for every Liouville vector field~$\Delta$. This implies 
    that $\Lie_{\Gamma} f = 0$ for all $\Gamma \in \Ker \omega$, which in these coordinates is equivalent to $f$ being a function only of $x^{\mu +1}, \dotsc, x^m$. Hence, we need to find a function $f(x^{\mu+1}, \dotsc,x^m)$ such that
    $$\sum_{j=\mu+1}^{m}b_j(x^{\mu+1},\dotsc,x^m)\fracp{f(x^{\mu+1}, \dotsc,x^m)}{x^j}=-1\,.$$
This is a first-order, inhomogeneous, linear PDE. 
Now,
$
\Delta_{00} = \sum_{j=\mu+1}^{m}b_j(x^{\mu+1},\dotsc,x^m)\fracp{}{x^j}
$
is a Liouville vector field, and is nowhere vanishing by definition.
Therefore there exist coordinates where it can be straightened out,
say
$\Delta_{00} = \partial/\partial z$.
Therefore the PDE becomes
$
\frac{\partial f}{\partial z} = -1 
$,
which proves the local existence of $f$.
\end{proof}

The condition $[\Gamma,\Gamma']\in \Ker \omega$, for all $\Gamma,\Gamma' \in \Ker \omega$ is equivalent to
$$\inn{\Gamma}\inn{\Gamma'}\d\omega = 0\,,$$
for all $\Gamma, \Gamma' \in \Ker \omega$.

Note that, for a doublet $(\tau,\omega)$ of even class,
the condition $[\Delta, \Gamma] \in \Ker \omega$ for all Liouville vector fields~$\Delta$ and all $\Gamma \in \Ker \omega$ is
$$\inn{[\Delta, \Gamma]} \omega = \Lie_{\Delta} (\inn{\Gamma}\omega) - \inn{\Gamma} (\Lie_{\Delta} \omega) = -\inn{\Gamma} (\inn{\Delta} \d \omega + \d\tau)= \inn{\Delta} \inn{\Gamma} \d \omega -\inn{\Gamma} \d \tau= 0\,.$$
If the first condition of Proposition~\ref{sufficient_cond_function} is satisfied, then it is enough to check this for a particular Liouville vector field $\Delta_0$. Indeed, as all the other Liouville vector fields can be written as $\Delta = \Delta_0 + \Gamma'$, with $\Gamma' \in \Ker \omega$, if we have $\inn{[\Delta_0, \Gamma]} \omega=0$, then
$$\inn{[\Delta, \Gamma]} \omega= \inn{\Delta_0+ \Gamma'} \inn{\Gamma} \d \omega -\inn{\Gamma} \d \tau=\inn{\Delta_0} \inn{\Gamma} \d \omega -\inn{\Gamma} \d \tau+ \inn{\Gamma'}\inn{\Gamma}\d\omega = \inn{\Gamma'}\inn{\Gamma}\d \omega=0\,.$$

In the particular case of precontact pair $(\eta,\d \eta)$ of even class, 
the hypotheses of the last proposition are always satisfied, 
as seen in Proposition~\ref{prop:preco_involutivity} and Remark~\ref{rem:Liou_involutivity}.
Therefore we have:

\begin{corollary}\label{existence_of_f}
    Let $\eta \in \Omega^{1}(M)$ be a nowhere-vanishing $1$-form of constant class $2r$ on $M$. Then, for any point $x\in M $, there exists a neighbourhood $V$ of $x$ and a nowhere-vanishing function $g\in \Cinfty(V)$, such that 
    $$\eta'=g\, \eta |_V$$
    is of constant class $2r-1$\,.
\end{corollary}
See~\cite{god_1969} for another proof of this result. 

Observe also that if the class of a pair $(\tau,\omega)$, with $\tau$ nowhere-vanishing, is even and equal to the dimension of the manifold, then the Liouville vector field is unique, $\Ker \omega$ is trivial, and thus the conditions of Proposition~\ref{sufficient_cond_function} are satisfied trivially.

\begin{example}
In~\cite{LGGMR_23} 
the study of the contact Lagrangian $L=vs$ led to the 1-form 
$$\eta= \d s - s \,\d q\,, $$
defined on $M = \Tan \R \times \R$ with coordinates $(q,v,s)$. 
It was noted that it does not possess Reeb vector fields.

Indeed, it is easy to see that $\eta$ has class $2$, since
    $$\d \eta = \d q \wedge \d s \qquad\text{and}\qquad \eta \wedge \d\eta = 0\,.$$
    Therefore, there exist Liouville vector fields for it. 
    They are $\Delta = \fracp{}{q}+b \fracp{}{v} + s\fracp{}{s}$, 
    where $b\in \Cinfty(M)$ is any arbitrary function.
    
    By Proposition~\ref{change_of_class_parity}, 
    the functions $f$ such that the $1$-form $e^f\eta$ has class~1 are those that satisfy 
    $\Lie_{\Delta }f = -1$, i.e.\
    $$\fracp{f}{q}+s\fracp{f}{s}+1=0 \qquad \text{and} \qquad \fracp{f}{v}=0\,.$$
    Let us take, for example, $f=-q$. Then, the $1$-form 
    $$e^{-q}\,\eta=e^{-q} \,\d s - e^{-q}s\,\d q$$
    has class 1. Indeed, we can find its Darboux coordinates (see Proposition~\ref{pc_darboux_odd}) easily, as
    $$e^{-q}\eta = \d (e^{-q}s)\,.$$
\end{example}

\subsection{Preservation of odd parity}
 
We can obtain similar results to characterize the functions which \emph{preserve} the odd parity of a doublet~$(\tau,\omega)$.

\begin{proposition} 
    Let $(\tau,\omega)$ be a doublet of class $2r+1$, and $f\in\Cinfty(M)$ a function. Then, the pair $(e^f\tau, e^f(\d f\wedge \tau+\omega))$ is a doublet of class $2r+1$ if and only if, for all Reeb vector fields $R\in\X(M)$ of the doublet $(\tau,\omega)$, the function satisfies 
    \begin{equation}\label{maintain_odd_class}
        (\inn{R} \d f)\,\tau\wedge\omega^r = \d f\wedge \omega^r\,.
    \end{equation}
\end{proposition}
\begin{proof}
    Let us denote $\widetilde \omega\coloneqq \d f \wedge \tau +\omega$. First, assume that $(e^f\tau, e^f\widetilde \omega)$ is a doublet of class $2r+1$. Then, necessarily, it satisfies $(e^f\widetilde \omega)^{r+1}=0$, by Proposition~\ref{conditions_odd_class}. Now, using Lemma~\ref{lemma_formula_ef_times_form}, we have
    $$(e^f\widetilde \omega)^{r+1} = e^{(r+1)f}((r+1)\, \d f \wedge\tau\wedge\omega^{r}+\omega^{r+1}) = e^{(r+1)f}((r+1)\, \d f \wedge\tau\wedge \omega^{r})\,,$$
    where we have used that $\omega^{r+1}=0$, which follows from $(\tau,\omega)$ having class $2r+1$, in the last equality. Necessarily, 
    $$\d f\wedge\tau\wedge\omega^r=0\,.$$
    And if we contract this last expression by any Reeb vector field $R \in \X(M)$ of the doublet $(\tau,\omega)$, we obtain
    \begin{align}
        \inn{R}(\d f\wedge\tau\wedge\omega^r) &= (\inn{R}\d f)\tau\wedge\omega^r - \d f \wedge\omega^r=0\,,
    \end{align}
    which proves the first implication.

    The converse follows directly from equation~\eqref{eq:after_lemma_eq} and Proposition~\ref{conditions_odd_class}, as $$(e^f\tau)\wedge (e^f\widetilde\omega)^{r}=e^{(r+1)f}(\tau\wedge \omega ^r)\neq 0\,, $$ and
    $$(e^f(\d f\wedge \tau+\omega))^{r+1} = e^{(r+1)f}((r+1)\, \d f \wedge\tau\wedge\omega^{r}+\omega^{r+1}) = e^{(r+1)f}((r+1)\, \d f \wedge\tau\wedge \omega^{r})\,,$$
    is equal to 0, because
    $$\d f \wedge\tau\wedge\omega^{r}=-\tau\wedge\d f\wedge\omega^r=-(\inn{R}\d f)\, \tau\wedge\tau\wedge\omega^r = 0\,,$$
    where we used the hypothesis in the last equality.
\end{proof}

It is well-known that if $(M,\eta)$ is a contact manifold, with $\dim M = 2n+1$, then $(M,g\eta)$ is still a contact manifold for any nowhere-vanishing function $g\in \Cinfty(M)$. As $(M,\eta)$ is a contact manifold if and only if the class of $\eta$ is $2n+1$, we have the following corollary:
\begin{corollary}
    Let $(M,\eta)$ be a contact manifold, with $\dim M = 2n+1$, and let $R$ be its Reeb vector field. Then, the equality 
    $$(\Lie_R f)\eta\wedge (\d \eta )^n = \d f \wedge (\d \eta) ^n$$
    holds for any function $f \in \Cinfty (M)$.
\end{corollary}
This corollary can also be proven by contracting the $(2n+2)$-form $\d f\wedge \eta \wedge (\d \eta)^n=0$ with the Reeb vector field.

\begin{proposition}\label{odd_class_stay}
        Let $(\tau,\omega)$ be a doublet of class $2r+1$, and $f\in\Cinfty(M)$. Then, the pair $(e^f\tau, e^f(\d f\wedge \tau+\omega))$ is a doublet of class $2r+1$ if, and only if, 
    $$\Lie_{\Gamma} f = 0\,,$$
for every $\Gamma \in \K_{(\tau , \omega)} = \Ker \tau \cap \Ker \omega$.
\end{proposition}
\begin{proof}
    Observe that we can rewrite equation~\eqref{maintain_odd_class}, as
$$((\inn{R}\d f) \tau -\d f)\wedge \omega ^r = 0\,.$$
As we are assuming that $\omega$ has rank $2r$, so $\omega^r \neq 0$, it is clear from Proposition~\ref{eq:even_conditions} and Theorem~\ref{thm:class-parity} that this is equivalent to

$$ 
((\inn{R} \d f)\tau - \d f ) \in (\Ker \omega )^{\circ}\,,
$$
for all Reeb vector fields $R \in \X(M)$ of the doublet $(\tau ,\omega)$.

 By definition of annihilator, this is equivalent to
$$(\inn{R} \d f)(\inn{\Gamma}\tau) - \inn{\Gamma} \d f= ((\inn{\Gamma} \tau)R-\Gamma)(f) = 0\,,$$
for all $\Gamma \in \Ker \omega$ and all Reeb vector fields $R \in \X(M)$. Last, note that 
$$(\inn{\Gamma} \tau ) R - \Gamma \in \K_{(\tau , \omega)} = \Ker \tau \cap \Ker \omega\,,$$
and also, clearly, every element of $X \in \K_{(\tau , \omega)}$ can be written in the form $(\inn{\Gamma}\tau) R -\Gamma$, with $\Gamma = -X$.
\end{proof}

As a direct consequence of this proposition and the Frobenius theorem,
we can state the following result,
which is analogous to Proposition~\ref{sufficient_cond_function}.
\begin{proposition}
Let $(\tau,\omega)$ be a doublet of class $2r+1$. 
If the characteristic distribution 
$\K _{(\tau,\omega)} = \Ker \tau \cap \Ker \omega$ 
is involutive, then, for any point $x\in M$, 
there exists a neighbourhood $V$ of $x$ and 
a nowhere-vanishing function $g\in \Cinfty(V)$ 
such that the pair $(g\tau,\d g\wedge \tau + g\omega)$ 
has constant class $2r+1$ on $V$.
\end{proposition}

\addcontentsline{toc}{section}{Conclusions and outlook}
\section*{Conclusions and outlook}

In this paper we have analysed in detail the geometric properties of pairs 
$(\tau,\omega)$
formed by a differential 1-form and a 2-form.
We introduced a definition of class for such pairs 
that extends the usual notion of class of a 1-form. 
To avoid degeneracies arising when the class varies, 
we mostly confined our investigation to pairs with constant class.
We showed how the parity of the class characterizes the existence of some distinguished vector fields, such as Reeb or Liouville vector fields. 
These objects turn out to be instrumental to understand the geometric structures associated with a pair $(\tau,\omega)$, 
such as the characteristic tensor, the extended characteristic distribution, and the extended 3-form. 

To demonstrate the unifying power of this approach, 
we showed how several standard geometric structures fit naturally into this framework. 
Specifically, we recovered the properties of contact, (pre)symplectic, and (pre)cosymplectic structures, as well as locally conformally symplectic manifolds and Hamiltonian systems.

We applied these results to 
clarify the notion of precontact form
and to define Hamiltonian dynamics associated with such structure. 
Finally, we investigated how conformal transformations 
may preserve or change the parity of the class of pairs $(\tau,\omega)$. 
We characterized these changes in terms of the associated characteristic distribution and the Liouville vector fields. 

We plan to apply this general framework to the study of singular Lagrangians, 
both for action-dependent and non-autonomous systems. 
It would also be interesting to analyse in detail different situations where the class is not constant, 
and how this may alter the behaviour of the associated geometrical objects.

\addcontentsline{toc}{section}{Acknowledgements}
\section*{Acknowledgments}

AMM acknowledges financial support from a predoctoral contract funded by Universitat Rovira i Virgili under grant 2025PMF-PIPF-14.
AMM and XRG acknowledge partial financial support from the Spanish Ministry of Science and Innovation grant RED2022-134301-T of AEI.

\bibliographystyle{abbrv}
{\small
\bibliography{references.bib}
}
\end{document}